\documentclass[12pt]{amsart}
\usepackage[american]{babel}
\usepackage{dsfont,mathtools,amssymb}
\usepackage{color}

\mathtoolsset{showonlyrefs,showmanualtags}

\theoremstyle{definition}
\newtheorem{theorem}{Theorem}[section]
\newtheorem{proposition}[theorem]{Proposition}
\newtheorem{lemma}[theorem]{Lemma}

\newtheorem{remark}[theorem]{Remark}
\newtheorem{definition}[theorem]{Definition}

\newtheorem{notation}[theorem]{Notation}
\newtheorem*{acknowledgement}{Acknowledgement}
\numberwithin{equation}{section}
\numberwithin{figure}{section}

\newcommand{\dx}{\mathrm{d}}
\newcommand{\e}{\mathrm{e}}
\newcommand{\E}{\mathds{E}}

\newcommand{\eps}{\varepsilon}

\renewcommand{\rho}{\varrho}
\renewcommand{\phi}{\varphi}

\newcommand{\h}{\frac{1}{2}}

\DeclareMathOperator{\var}{var}
\DeclareMathOperator{\cov}{cov}

\DeclareMathOperator{\dist}{dist}
\DeclareMathOperator{\osc}{osc}
\DeclareMathOperator{\Id}{Id}

\DeclareMathOperator{\supp}{supp}
\DeclareMathOperator{\PI}{PI}

\usepackage{hyperref}
\hypersetup{pdftex}

\title[Otto-Reznikoff revisited]{The approach of Otto-Reznikoff revisited}

\date{February 10, 2014}

\subjclass[2000]{Primary 60K35; secondary 82B20; 82C26.}

\keywords {lattice systems, continuous spin, logarithmic Sobolev inequality, decay of correlations}

\author{Georg Menz }
\address{Georg Menz\\ Stanford University}
\email{gmenz@stanford.edu}

\begin{document}
\begin{abstract}
In this article we consider a lattice system of unbounded continuous spins. Otto \& Reznikoff used the two-scale approach to show that exponential decay of correlations yields a logarithmic Sobolev inequality (LSI) with uniform constant in the system size. We improve their statement by weakening the assumptions. For the proof a more detailed analysis based on two new ingredients is needed. The two new ingredients are a covariance estimate and a uniform moment estimate. We additionally provide a comparison principle for covariances showing that the correlations for the conditioned Gibbs measures are controlled by the correlations of the original Gibbs measure with ferromagnetic interaction. The latter simplifies the application of the main result. As an application, we show how decay of correlations combined with the uniform LSI yields the uniqueness of the infinite-volume Gibbs measure, generalizing a result of Yoshida from finite-range to infinite-range interaction.

\end{abstract}

\maketitle

\section{Introduction and main results} 

We consider a lattice system of unbounded and continuous spins on the $d$-dimensional lattice $\mathds{Z}^d$. The formal Hamiltonian $H: \mathds{R}^{\mathds{Z}^d} \to \mathds{R}$ of the system is given by  
\begin{equation}
  \label{e_d_Hamiltonian}
	H(x) = \sum_{i \in \mathds{Z}^d } \psi_i(x_i) + \frac{1}{2} \sum_{i,j \in \mathds{Z}^d} M_{ij} x_i x_j.  
\end{equation}
We assume that the single-site potentials $\psi_i : \mathds{R} \to \mathds{R}$ are smooth and perturbed convex. This means that there is a splitting $\psi_i= \psi_i^c + \psi_i^b$ such that for all $i \in \mathds{Z}^d$ and $z \in \mathds{R}$
\begin{equation}\label{e_cond_psi}
  (\psi_i^c)'' (z) \geq 0 \qquad \mbox{and} \qquad |\psi_i^b (z)| + | ( \psi_i^b)' (z) | \lesssim 1.
\end{equation}
Here, we used the convention (see Definition~\ref{def:dep} below for more details)
\begin{equation*}
  a \lesssim b \qquad :\Leftrightarrow \mbox{there is a uniform constant $C>0$ such that $a \leq C b$}.
 \end{equation*}
Moreover, we assume that 
\begin{itemize}
\item  the interaction is symmetric i.e.~
  \begin{equation}
    \label{e_ass_sym}
    M_{ij}=M_{ji} \qquad \mbox{ for all $i, j \in \mathds{Z}^d$,}
  \end{equation}
  
\item and the matrix $M= (M_{ij})$ is strictly diagonal dominant i.e.~for some $\delta > 0$ it holds for any $i \in \mathds{Z}^d$
\begin{equation}\label{e_strictly_diag_dominant}
  \sum_{j \in \mathds{Z}^d, j \neq i} |M_{ij}| + \delta \le M_{ii}. 
\end{equation}
\end{itemize}

\begin{notation}
Let $S\subset \mathds{Z}^d$ be an arbitrary subset of $\mathds{Z}^d$. For convenience, we write $x^S$ as a shorthand for $(x_i)_{i \in S}$. 
\end{notation}
\begin{definition}[Tempered spin-values]
Given a finite subset $\Lambda\subset \mathds{Z}^d$, we call the spin values $x^{\mathds{Z}^d \backslash \Lambda}$ tempered, if for all $i \in \Lambda$
\begin{equation*}
  \sum_{j \in \mathds{Z}^d \backslash {\Lambda}} |M_{ij}| \ |x_j| < \infty.
\end{equation*} 
\end{definition}
\begin{definition}[Finite-volume Gibbs measure]
  Let $\Lambda$ be a finite subset of the lattice $\mathds{Z}^d$ and let $x^{\mathds{Z}^d \backslash \Lambda}$ be a tempered state. We call the measure $\mu_{\Lambda}( dx^{\Lambda})$ finite-volume Gibbs measure associated to the Hamiltonian $H$ with boundary values $x^{\mathds{Z}^d \backslash \Lambda}$, if it is a probability measure on the space $\mathds{R}^{\Lambda}$ given by the density
\begin{equation}
  \label{e_d_Gibbs_measure}
	\mu_{\Lambda}(dx^\Lambda) = \frac{1}{Z_{\mu_\Lambda}} \e^{-H(x^\Lambda,x^{\mathds{Z}^d \backslash \Lambda} )} \dx x^\Lambda .
\end{equation} 
Here, $Z_{\mu_\Lambda}$ denotes the normalization constant that turns $\mu_{\Lambda}$ into a probability measure. If there is no ambiguity, we also may write $Z$ to denote the normalization constant of a probability measure. We also used the short notation
\begin{align*}
  H(x^\Lambda,x^{\mathds{Z}^d \backslash \Lambda} ) = H(x) \quad \mbox{with} \quad x= (x^\Lambda,x^{\mathds{Z}^d \backslash \Lambda}). 
\end{align*}
 Note that $\mu_\Lambda$ depends on the spin values $x^{\mathds{Z}^d \backslash \Lambda}$ outside of the set $\Lambda$.
\end{definition}

The main object of study in this article is the question if the finite-volume Gibbs measure $\mu_{\Lambda}$ satisfies a logarithmic Sobolev inequality~(LSI).
\begin{definition}[LSI] \label{d_LSI}
  Let $X$ be a Euclidean space. A Borel probability measure $\mu$ on $X$ satisfies the LSI with constant $\varrho>0$, if for all smooth functions $f \geq 0$ 
  \begin{equation}\label{e_definition_of_LSI}
  \int f \log f \ d \mu - \int f  d\mu  \log \left( \int f  d\mu \right) \leq \frac{1}{2 \varrho} \int \frac{|\nabla f|^2}{f} d\mu.  
  \end{equation}
Here, $\nabla$ denotes the gradient determined by the Euclidean structure of~$X$.
\end{definition}
The LSI yields by linearization the Poincar\'e inequality (PI) (see for example~\cite{L}).
\begin{definition}[PI]\label{d_SG}
   Let $X$ be a Euclidean space. A Borel probability measure $\mu$ on $X$ satisfies the PI with constant $\varrho>0$, if for all smooth functions~$f$
\[
   \var_{\mu} (f): = \int \left(f - \int f d \mu \right)^2 d \mu \leq \frac{1}{\varrho} \int |\nabla f|^2 d\mu.
\]   
Here, $\nabla$ denotes the gradient determined by the Euclidean structure of~$X$.
\end{definition}

The LSI was originally introduced by Gross \cite{Gross}. It can be used as a powerful tool for studying spin systems. The LSI implies exponential convergence to equilibrium of the naturally associated conservative diffusion process. The rate of convergence is given by the LSI constant $\varrho$ (cf.~\cite[Chapter~3.2]{Roy07}). At least in the case of finite-range interaction, independence from the system size of the LSI constant of the local Gibbs state directly yields the uniqueness of the infinite-volume Gibbs state (cf.~\cite{Roy07,Yos_2,Zitt}). \medskip

In the literature, there are several results known that connect the decay of spin-spin correlations to the validity of a LSI uniform in the system size \cite{StrZeg,StrZeg2,Zeg96,Yos99,Yos01,B-H1}. This means that a static property of the equilibrium state of the system is connected to a dynamic property namely the relaxation to the equilibrium. We refer the reader to the Section~2.2.~of the article of Otto \& Reznikoff~\cite{OR07}, which gives a nice overview and discussion on the results in the literature. Otto \& Reznikoff used the two-scale criterion for the LSI (cf.~\cite[Theorem~1]{OR07} or \cite[Theorem 3]{GORV}) to deduce the following statement:\medskip
\begin{theorem}[\mbox{\cite[Theorem 3]{OR07}}]\label{p_OR_original}
Consider the formal Hamiltonian $H:\mathds{R}^{\mathds{Z}^d} \to \mathds{R} $ given by~\eqref{e_d_Hamiltonian}. Assume that the single site potentials $\psi_i=\psi$ are given by a function of the form
\begin{align}
  \label{e_single_site_potential_otto}
  \psi(z) = \frac{1}{12} z^4 + \psi^b (z) \qquad \mbox{with} \quad | \frac{d^2}{dz^2} \psi^b (z)|\leq C.
\end{align}
Assume that the interaction is symmetric i.e. $M_{ij}= M_{ji}$ and has zero diagonal i.e. $M_{ii}=0$. Consider a subset~$\Lambda_{\mathrm{tot}} \subset \mathbb{Z}^d$. We assume the uniform control:
\begin{align}
  \label{e_decay_inter_Otto}
  |M_{ij}| \lesssim  \exp \left( - \frac{|i-j|}{C} \right)
  \end{align}
for $i,j \in \Lambda$ and
\begin{align}
  \label{e_decay_corr_Otto}
|  \cov_{\mu_{\Lambda}} (x_i,x_j)| \lesssim \exp \left( - \frac{|i-j|}{C} \right)
\end{align}
 uniformly in $\Lambda \subset \Lambda_{\mathrm{tot}}$ and $i,j \in \Lambda$. Here, $\mu_{\Lambda}$ denotes the finite-volume Gibbs measures $\mu_{\Lambda}$ given by~\eqref{e_d_Gibbs_measure}. \newline
Then the finite-volume Gibbs measure $\mu_{\Lambda_{\mathrm{tot}}}$ satisfies the LSI with constant $\varrho>0$ depending only on the constant $C>0$ in~\eqref{e_single_site_potential_otto}, \eqref{e_decay_inter_Otto}, and~\eqref{e_decay_corr_Otto}.
\end{theorem}

The most important feature of Theorem~\ref{p_OR_original} is that the LSI constant $\varrho$ is independent of the system size $|\Lambda_{\mathrm{tot}}|$ and of the spin values~$x^{\mathds{Z}^d\backslash \Lambda_{\mathrm{tot}}}$ outside of~$\Lambda_{\mathrm{tot}}$. The advantage of Theorem~\ref{p_OR_original} over existing results connecting a decay of correlations to a uniform LSI is that it can deal with infinite-range interaction (cf.~\cite{StrZeg,StrZeg2,Zeg96,Yos99,Yos01,B-H1}). However, Theorem~\ref{p_OR_original}  calls for some technical improvements. The main result of this article is the following generalized version of Theorem~\ref{p_OR_original}:
\begin{theorem}[Generalization of \mbox{\cite[Theorem 3]{OR07}}]\label{p_mr_OR}
 Assume that the formal Hamiltonian $H:\mathds{R}^{\mathds{Z}^d} \to \mathds{R} $ given by~\eqref{e_d_Hamiltonian} satisfies the Assumptions~\eqref{e_cond_psi}~-~\eqref{e_strictly_diag_dominant}. Let $\Lambda_{\mathrm{tot}} \subset \mathds{Z}^d$ be an arbitrary, finite subset of the lattice $\mathds{Z}^d$.\newline 
Assume the following decay of interactions and correlations: For some $\alpha>0$ it holds
  \begin{equation}\label{e_cond_inter_alg_decay_OR}
  |M_{ij}| \lesssim \frac{1}{|i-j|^{d+ \alpha}}
  \end{equation}
 uniformly in $i,j \in \Lambda_{\mathrm{tot}}$ and
\begin{equation}\label{e_cond_alg_decay_OR}
 |  \cov_{\mu_{\Lambda}} (x_i,x_j)|  \lesssim \frac{1}{|i-j|^{d + \alpha}}
\end{equation}
 uniformly in $\Lambda \subset \Lambda_{\mathrm{tot}}$, and $i,j \in \Lambda$. Here, $\mu_{\Lambda}$ denote the finite-volume Gibbs measures given by (cf.~\eqref{e_d_Gibbs_measure}).\newline
Then the finite-volume Gibbs measure $\mu_{\Lambda_{\mathrm{tot}}}$ satisfies the LSI with a constant $\varrho>0$ depending only on the constant in~\eqref{e_cond_psi},~\eqref{e_strictly_diag_dominant}, \eqref{e_cond_alg_decay_OR} and \eqref{e_cond_inter_alg_decay_OR}.
\end{theorem}

Theorem~\ref{p_mr_OR} improves Theorem~\ref{p_OR_original} in two ways:\newline
Note that Theorem~\ref{p_OR_original} needs an exponential decay of interaction and spin-spin correlations. However, analyzing the proof of~\cite[Theorem~3]{OR07} one sees that the exponential decay is only needed to guarantee that certain sums are summable. Therefore this assumption can be weakened to algebraically decaying interaction and spin-spin correlations. Of course now, the order of the algebraic decay depends on the dimension of the underlying lattice to guarantee summability. \smallskip

The second improvement is more subtle. Theorem~\ref{p_OR_original} needs a special structure on the single-site potentials $\psi_i$. Namely, the single-site potentials $\psi_i$ have to be perturbed quartic in the sense of~\eqref{e_single_site_potential_otto}. Analyzing the proof of~\cite[Theorem~3]{OR07} shows that the argument does not rely on a quartic potential~$\psi_i^c$. For the argument of Otto \& Reznikoff it would be sufficient to have a perturbation of a strictly-superquadratic potential~i.e. 
\begin{equation}
  \label{e_super_quadratic}
  \liminf_{|x|\to \infty} \frac{d^2}{dx^2} \psi_i^c (x) \to \infty.
\end{equation}
The condition~\eqref{e_super_quadratic} on the single-site potential $\psi_i$ is widespread and accepted in the literature on the uniform LSI (cf.~for example~\cite{Yos01,Yos_2,ProSco}). \smallskip

However, a result by Zegarlinski~\cite[Theorem~4.1.]{Zeg96} indicates that the condition~\eqref{e_super_quadratic} is not necessary for deducing a uniform LSI. Zegarlinski deduced in~\cite[Theorem~4.1.]{Zeg96} the uniform LSI for the finite-volume Gibbs measure $\mu_\Lambda$ given by~\eqref{e_d_Gibbs_measure} on an one-dimensional lattice $\Lambda_{\mathrm{tot}} \subset \mathds{Z}$ with finite-range interaction.  For Zegarlinski's argument it is sufficient that the single-site potentials $\psi_i$ satisfy the conditions~\eqref{e_cond_psi} and~\eqref{e_strictly_diag_dominant}, which is strictly weaker than the condition~\eqref{e_super_quadratic} (for a proof of this statement we refer the reader to \cite[Proof of Lemma~1]{OR07}). In Theorem~\ref{p_mr_OR} we show that the conditions~\eqref{e_cond_psi} and~\eqref{e_strictly_diag_dominant} are in fact also sufficient for the Otto-Reznikoff approach. 

\begin{remark}\label{r_linear_term}
 Note that the structural assumptions~\eqref{e_cond_psi}~-~\eqref{e_strictly_diag_dominant} on the Hamiltonian $H$ are invariant under adding a linear term like 
  \begin{equation*}
    \sum_{i \in \mathds{Z}} x_ib_i
  \end{equation*}
for arbitrary $b_i \in \mathds{R}$. Therefore the the LSI constant of Theorem~\ref{p_mr_OR} is invariant under adding a linear term to the Hamiltonian. Such a linear term can be interpreted as a field acting on the system. If the coefficients $b_i$ are chosen randomly, one calls the linear term random field.
\end{remark}

Let us discuss what are the ingredients to weaken the structural assumptions on the single-site potential from the condition~\eqref{e_single_site_potential_otto} to the condition~\eqref{e_cond_psi} and~\eqref{e_strictly_diag_dominant}. Analyzing the proof of Otto \& Reznikoff, it all boils down to understanding the structure of the Hamiltonian of the marginals of the finite-volume Gibbs measure conditioned on the spin values of some set $S \subset \Lambda_{\mathrm{tot}}$ (cf.~\cite[Lemma 2, Lemma 3 and Lemma 4]{OR07} or see Section~\ref{s_LSI}). Because our structural assumptions ~\eqref{e_cond_psi} and~\eqref{e_strictly_diag_dominant} on the single-site potentials are weaker, our proof needs new ingredients and more detailed arguments compared to~\cite{OR07}. \medskip

The first new ingredient in the proof of Theorem~\ref{p_mr_OR} is the covariance estimate of Proposition~\ref{p_algebraic_decay_correlations}. With this estimate it is possible to deduce algebraic decay of correlations, provided the interactions $M_{ij}$ also decay algebraically and the nonconvex perturbation~$\psi_i^b$ is small enough.\medskip

The second new ingredient in the proof of Theorem~\ref{p_mr_OR} is a uniform estimate of $\var_{\mu_{\Lambda}} (x_i)$ (see Lemma~\ref{p_est_var_ss}), which we reduce to a moment estimate due to Robin Nittka (cf.~\cite[Lemma 4.2]{MN} and Lemma~\ref{lem:moments}). The full proof of Theorem~\ref{p_mr_OR} is given in Section~\ref{s_LSI}.\medskip

However, Theorem~\ref{p_mr_OR} still calls for further improvements. Note that in the condition~\eqref{e_cond_alg_decay_OR} of Theorem~\ref{p_mr_OR} one needs to check the decay of correlations for all finite-volume Gibbs measures $\mu_\Lambda$ with $\Lambda \subset \Lambda_{\mathrm{tot}}$. Even if this is a very common assumption (see for example \cite[Condition (DS3)]{Yos01}) it may be a bit tedious to verify. Instead of the strong condition~\eqref{e_cond_alg_decay_OR}, one would like to have a weak condition like the one used for discrete spins in~\cite{MaOl}. The main difference between the weak and the strong condition for the decay of correlations is that in the weak condition it suffices to show that for a sufficiently large box $\Lambda$ the correlations decay nicely. The main advantage of the weak condition is that one does not have to control the decay of correlations for all growing subsets $\Lambda \to \mathds{Z}^d$. Therefore, the weak condition is easier to verify by experiments. Unfortunately, we cannot get rid of the strong decay of correlations condition~\eqref{e_cond_alg_decay_OR} in the Otto-Reznikoff approach. However, we show how verifying the strong decay of correlations condition~ \eqref{e_cond_alg_decay_OR} can be simplified by two comparison principles.\smallskip

The first comparison principle (see Lemma~\ref{p_comparison_covariances} below) shows that in the case of ferromagnetic interaction (i.e.~$M_{ij}<0$ for all $i,j \in \Lambda_{\mathrm{tot}}$) the correlations of a smaller system are controlled by correlations of the larger system.

\begin{lemma}\label{p_comparison_covariances}
Assume that the formal Hamiltonian $H:\mathds{R}^{\mathds{Z}^d} \to \mathds{R} $ given by~\eqref{e_d_Hamiltonian} satisfies the Assumptions~\eqref{e_cond_psi}~-~\eqref{e_strictly_diag_dominant}.  Additionally, assume that the interactions are ferromagnetic i.e.~$M_{i,j} \leq 0$ for $i \neq j$. \newline 
 For arbitrary subsets $\Lambda \subset \Lambda_{\mathrm{tot}} \subset \mathds{Z}^d$, we consider the finite-volume Gibbs measure $\mu_{\Lambda}$ and $\mu_{\Lambda_{\mathrm{tot}}}$ with the same tempered state $x^{\mathds{Z}^d \backslash \Lambda_{\mathrm{tot}}}$. Then it holds for any $i,j \in \Lambda$
 \begin{align} \label{e_comp_small_big_corr}
 \cov_{\mu_{\Lambda}} (x_i,x_j) \leq \cov_{\mu_{\Lambda_{\mathrm{tot}}}}(x_i,x_j) .
 \end{align}
\end{lemma}
The proof of Lemma~\ref{p_comparison_covariances} is given in Section~\ref{s_covariance_estimate}. The second comparison principle is rather standard. It states that correlations of a non-ferromagnetic system are controlled by the correlations of the associated ferromagnetic system: 

\begin{lemma}\label{p:attractive_interact_dominates}
Assume that the formal Hamiltonian $H:\mathds{R}^{\mathds{Z}^d} \to \mathds{R} $ given by~\eqref{e_d_Hamiltonian} satisfies the Assumptions~\eqref{e_cond_psi}~-~\eqref{e_strictly_diag_dominant}.  
Let $\mu_{\Lambda}$ denote the finite-volume Gibbs measure given by~\eqref{e_d_Gibbs_measure}. Additionally, consider the corresponding finite-volume Gibbs measure~~$\mu_{\Lambda,|M|}$ with attractive interaction i.e.~the associated formal Hamiltonian is given by 
\begin{equation*}
 %  \label{e_d_Hamiltonian}
	H(x) = \sum_{i \in \mathds{Z}^d } \psi_i(x_i) - \frac{1}{2} \sum_{i,j \in \mathds{Z}^d} |M_{ij}| x_i x_j.  
\end{equation*}
Then it holds that for any $i,j \in \Lambda$
\begin{equation}
  \label{eq:covariance_domination}
  | \cov_{\mu_{\Lambda}} (x_i,x_j)  | \leq \cov_{\mu_{\Lambda,|M|}} (x_i,x_j) .
\end{equation}
\end{lemma}
We do not state the proof of the last lemma. One can find the proof for example in a recent work by Robin Nittka and the author. The proof follows the argument of~\cite{HorMor} for discrete spins (see~\cite[Lemma 2.1.]{MN}).

\begin{remark}
Usually, one considers finite-volume Gibbs measures for some inverse temperature $\beta >0$ i.e.
  \begin{equation*}
    	\mu_\Lambda (d x^\Lambda ) = \frac{1}{Z_\mu} \e^{- \beta H(x^\Lambda, x^{\mathds{Z} \backslash \Lambda})} \dx x \qquad \mbox{for } x^\Lambda \in \mathds{R}^{\Lambda}.
  \end{equation*}
This case is also contained in the main results of the article, because the Hamiltonian $\beta H$ still satisfies the structural Assumptions~\eqref{e_cond_psi}~-~\eqref{e_strictly_diag_dominant}. Of course, the LSI constant of Theorem~\ref{p_mr_OR} would depend on the inverse temperature~$\beta$. 
\end{remark}

\begin{remark}
  Because we assume that the matrix $M= (M_{ij})$ is strictly diagonal dominant (cf.~\eqref{e_strictly_diag_dominant}), the full single-site potential 
  \begin{equation*}
    \psi_i(x_i) + M_{ii} x_i^2 = M_{ii} x_i^2 + \psi_i^c(x_i) + \psi_i^b (x_i)
  \end{equation*}
is perturbed strictly-convex. We want to note that this is the same structural assumption as used in the article~\cite{MO}. 
\end{remark}

Let us turn to an application of Theorem~\ref{p_mr_OR}. We will show how the decay of correlations condition~\eqref{e_cond_alg_decay_OR} combined with the uniform LSI of Theorem~\ref{p_mr_OR} yields the uniqueness of the infinite-volume Gibbs measure. The statement that a uniform LSI yields the uniqueness of the Gibbs state is already known from the case of finite-range interaction (cf. for example~\cite{Yos_2}, the~conditions (DS1), (DS2), and (DS3) in~\cite{Yos01}). The related arguments of~\cite{Roy07},~\cite{Zitt}, and~\cite{Yos01} are based on semigroup properties of an associated diffusion process. Though the semigroup probably may work in the case of infinite-range interaction, we follow a more straightforward approach to deduce the uniqueness of the Gibbs measure. Before we formulate the precise statement (see Theorem~\ref{p_unique_Gibbs} below), we specify the notion of an infinite-volume Gibbs measure.
\begin{definition}[Infinite-Volume Gibbs measure]
  Let $\mu$ be a probability measure on the state space $\mathds{R}^{\mathds{Z}^d}$ equipped with the standard product Borel sigma-algebra. For any finite subset $\Lambda \subset \mathds{Z}^d$ we decompose the measure $\mu$ into the conditional measure $\mu(dx^\Lambda| x^{\mathds{Z}^d \backslash \Lambda})$ and the marginal $\bar \mu (d x^{\mathds{Z}^d \backslash \Lambda})$. This means that for any test function $f$ it holds
\begin{equation*}
  \int f(x) \mu (dx) = \int \int f(x) \mu(dx^\Lambda| x^{\mathds{Z}^d \backslash \Lambda}) \bar \mu (d x^{\mathds{Z}^d \backslash \Lambda}).
\end{equation*}
We say that the measure $\mu$ is the infinite-volume Gibbs measure associated to the Hamiltonian $H$, if the conditional measures $\mu(dx^\Lambda| x^{\mathds{Z}^d \backslash \Lambda})$ are given by the finite-volume Gibbs measures $\mu_{\Lambda}(dx^\Lambda)$ defined by~\eqref{e_d_Gibbs_measure} i.e.
\begin{equation*}
  \mu(dx^\Lambda| x^{\mathds{Z}^d \backslash \Lambda}) = \mu_\Lambda (dx^\Lambda).
\end{equation*} 
The equations of the last identity are also called Dobrushin-Lanford-Ruelle (DLR) equations. 
\end{definition}
The precise statement connecting the decay of correlations with the uniqueness of the infinite-volume Gibbs measure is:
\begin{theorem}[Uniqueness of the infinite-volume Gibbs measure]\label{p_unique_Gibbs}
 Under the same assumptions as in Theorem~\ref{p_mr_OR}, there is at most one unique Gibbs measure $\mu$ associated to the Hamiltonian $H$ satisfying the uniform bound
\begin{equation}~\label{e_sup_moment}
    \sup_{i \in \mathds{Z}^d} \int (x_i)^2 \mu (dx) < \infty.
\end{equation}
 \end{theorem}
The moment condition~\eqref{e_sup_moment} in Theorem~\ref{p_unique_Gibbs} is standard in the study of infinite-volume Gibbs measures (see for example~\cite{BHK82} and~\cite[Chapter~4]{Roy07}). It is relatively easy to show that the condition~\eqref{e_sup_moment} is invariant under adding a bounded random field to the Hamiltonian~$H$ (cf.~Remark~\ref{r_linear_term}). \smallskip

Theorem~\ref{p_unique_Gibbs} is one of the \emph{well-known} statements for which it is hard to find a proof. Therefore we state the proof in full detail in the Appendix~\ref{s_decay_and_uniqueness}. The argument does not need that the finite-volume Gibbs measures $\mu_{\Lambda}$ satisfy a uniform LSI. It suffices that the finite-volume Gibbs measures $\mu_{\Lambda}$ satisfy a uniform PI, which is a weaker condition then the LSI (see~Definition~\ref{d_SG}).   \medskip

We also want to note that the main results of this article, namely Theorem~\ref{p_mr_OR} and Theorem~\ref{p_unique_Gibbs} were applied in~\cite{MN} to deduce a uniform LSI and the uniqueness of the infinite-volume Gibbs measure of a one-dimensional lattice system with long-range interaction, generalizing Zegarlinsk's result~\cite[Theorem~4.1.]{Zeg96} to interactions of infinite range.

\begin{remark}
  In this article, we do not show the existence of an infinite-volume Gibbs measure. However, the author of this article believes that under the assumption~\eqref{e_sup_moment} the existence should follow by an compactness argument similarly to the one used in~\cite{BHK82}.   
\end{remark}

In order to avoid confusion, let us make the notation $a \lesssim b$ from above precise.
\begin{definition}\label{def:dep}
	We will use the notation $a \lesssim b$ for quantities $a$ and $b$
	to indicate that there is a constant $C \ge 0$
	which depends only on a lower bound for $\delta$ and upper bounds for $|\psi_i^b|$, $|(\psi_i^b)'|$, and $\sup_i \sum_{j \in \mathds{Z}^d} |M_{ij}|$ such that $a \le C b$. In the same manner, if we assert the existence of certain constants, they may freely depend on the above mentioned quantities, whereas all other dependencies will be pointed out.
\end{definition}

We close the introduction by giving an outline of the article.\smallskip
\begin{itemize}
% \item In Section~\ref{s_bl_covariance_estimate}, we present the new covariance estimate of Theorem~\ref{p_decay_of_correlations_linear}. We also discuss how it can be used to deduce decay of correlations for weakly non-convex perturbed Gibbs measures. The proof of Theorem~\ref{p_decay_of_correlations_linear} emerged from discussions with Felix Otto. 
\item In Section~\ref{s_covariance_estimate}, we prove Lemma~\ref{p_comparison_covariances}. This contains the comparison principle for covariances of smaller systems to larger systems.
\item In Section~\ref{s_LSI}, we consider the generalization of Theorem~\ref{p_OR_original} and give the proof of Theorem~\ref{p_mr_OR}. 
\item In the Appendix~\ref{s_decay_and_uniqueness}, we consider the uniqueness of the infinite-volume Gibbs measure and give the proof of Theorem~\ref{p_unique_Gibbs}.
\item In the Appendix~\ref{s_BE_HS} we state some well-known facts about the LSI and the PI. 
\end{itemize}

\section{Comparing covariances of a smaller system to covariances to a bigger system: Proof of Lemma~\ref{p_comparison_covariances}}\label{s_covariance_estimate}

The proof of Lemma~\ref{p_comparison_covariances} uses an idea of Sylvester of expanding the exponential function~\cite{Sylvester}. Sylvester used this idea to give a simple unified derivation of a bunch of correlation inequalities for ferromagnets.

\begin{proof}[Proof of Lemma~\ref{p_comparison_covariances}]
  We fix the spin values $m_i$, $i \in \Lambda_{\mathrm{tot}} \backslash \Lambda$. Recall that in our notations $\mu_\Lambda$ coincides with the conditional measure 
  \begin{align*}
    \mu_{\Lambda}(d x^{\Lambda}) = \mu_{\Lambda_{\mathrm{tot}}}(dx^\Lambda | m^{\Lambda_{\mathrm{tot}} \backslash \Lambda }).
  \end{align*}
We introduce the auxiliary Hamiltonian $H_\alpha$, $\alpha >0$, by the formula 
  \begin{align*}
    H_\alpha (x) = H(x) + \alpha \sum_{i \in \Lambda_{\mathrm{tot}}\backslash \Lambda }(x_i - m_i)^2.
  \end{align*}
We denote by $\mu_{\alpha}$ the associated Gibbs measure active on the sites $\Lambda_{\mathrm{tot}}$. The measure $\mu_{\alpha}$ is given by the density
\begin{align*}
  \mu_{\alpha} (dx) = \frac{1}{Z} \ \exp \left( - H_{\alpha} (x) \right) \ d x \qquad \mbox{for } x \in \mathbb{R}^{\Lambda_{\mathrm{tot}}}.
\end{align*}
Note that the measure $\mu_{\alpha}$ interpolates between the measure $\mu_{\Lambda}$ and $\mu_{\Lambda_{\mathrm{tot}}}$ in the sense that $\mu_0 = \mu_{\Lambda_{\mathrm{tot}}}$ and for any integrable function $f: \mathbb{R}^{\Lambda} \to \mathbb{R}$
\begin{align*}
 \lim_{\alpha \to \infty} \int f(x^\Lambda)  \mu_\alpha (dx^{\Lambda_{\mathrm{tot}}}) = \int f(x^\Lambda) \mu(dx^\Lambda | m^{\Lambda_{\mathrm{tot}} \backslash \Lambda}).
\end{align*}
So we formally have $\mu_{\infty} = \mu_{\Lambda}$
Therefore it also holds for $i, j \in \Lambda$
\begin{align*}
   \lim_{\alpha \to \infty} \cov_{\mu_\alpha} (x_i, x_j) = \cov_{\mu_{\Lambda}} (x_i, x_j) 
\end{align*}
This yields by the fundamental theorem of calculus that
\begin{align*}
  \cov_{\mu_\infty} (x_i, x_j) - \cov_{\mu_0} (x_i, x_j) = \int_0^\infty \frac{d}{d \alpha} \cov_{\mu_\alpha} (x_i, x_j).
  \end{align*}
We will now show that $$\frac{d}{d \alpha} \cov_{\mu_\alpha} (x_i, x_j) < 0,$$ which yields the statement of Lemma~\ref{p_comparison_covariances}.\newline 
Indeed, direct calculation shows that
\begin{align*}
 & \frac{d}{d \alpha} \cov_{\mu_\alpha} (x_i, x_j)\\
 & \quad =   \frac{d}{d \alpha}  \left( \int x_i x_j \mu_{\alpha} - \int x_i \mu_{\alpha} \int x_j \mu_{\alpha}  \right) \\
  & \quad = - \cov_{\mu_\alpha} \left(  x_i x_j  - \int x_i \mu_{\alpha} \int x_j \mu_{\alpha} , \sum_{l \in \Lambda_{\mathrm{tot}} \backslash \Lambda }(x_l - m_l)^2 \right) \\
    & \quad = - \cov_{\mu_\alpha} \left(  x_i x_j , \sum_{l \in \Lambda_{\mathrm{tot}}\backslash \Lambda }(x_l - m_l)^2 \right).
\end{align*}
We will show now that 
\begin{align}
  \cov_{\mu_\alpha} \left(  x_i x_j  , \sum_{l \in \Lambda_{\mathrm{tot}} \backslash \Lambda }(x_l - m_l)^2 \right) \geq 0. \label{e_cond_covariances_crucial_estimate}
\end{align}
For this purpose, we follow the method by Sylvester~\cite{Sylvester} of expanding the interaction term. Recall that this method is also used to show for example that 
\begin{align*}
  \cov_{\mu_\alpha} \left(  x_i , x_j \right) \geq 0,
\end{align*}
provided the interactions are ferromagnetic. By doubling the variables we get
\begin{align*}
&    \cov_{\mu_\alpha} \left(  x_i x_j , \sum_{l \in \Lambda_{\mathrm{tot}}\backslash \Lambda }(x_l - m_l)^2 \right) \\
& = \int ( x_i x_j - \tilde x_i \tilde x_j ) \sum_{l \in \Lambda_{\mathrm{tot}} \backslash \Lambda }(x_l - m_l)^2 \mu_\alpha (dx) \ \mu_\alpha (d \tilde x) \\
&=  \frac{1}{Z^2}\int ( x_i x_j - \tilde x_i \tilde x_j ) \sum_{l \in \Lambda_{\mathrm{tot}} \backslash \Lambda }(x_l - m_l)^2 \exp (- H_\alpha (x) - H_\alpha (\tilde x)) dx d \tilde x
\end{align*}
Because the partition function $Z>0$ is positive, the sign of the covariance is determined by the integral on the right hand side of the last identity.
We change variables according to $x_i=  (p_i + q_i)$ and $\tilde x_i =  (p_i - q_i)$ and get
\begin{align}
&  \int ( x_i x_j - \tilde x_i \tilde x_j ) \sum_{l \in \Lambda_{\mathrm{tot}\backslash \Lambda }}(x_l - m_l)^2 \exp (- H_\alpha (x) - H_\alpha (\tilde x)) dx d \tilde x \notag \\
& \quad  = C \int ( (p_i + q_i) (p_j + q_j) - (p_i - q_i) (p_j - q_j) ) \notag \\
& \qquad \times \sum_{l \in \Lambda_{\mathrm{tot}}\backslash \Lambda }((p_l + q_l) - m_l)^2 \exp (- H_\alpha (p-q) - H_\alpha (p+q)) dp dq \notag \\
& \quad = C \int ( 2 p_i  q_j  + 2 q_i p_j)  \notag \\
& \qquad \times \sum_{l \in \Lambda_{\mathrm{tot}}\backslash \Lambda }((p_l + q_l) - m_l)^2 \exp (- H_\alpha (p-q) - H_\alpha (p+q)) dp dq  \label{e_integral_to_expand}
\end{align}
where $C>0$ is the constant from the transformation. Straightforward calculation reveals
\begin{align}
&   \tilde H_\alpha (p,q)  = H_\alpha (p-q) + H_\alpha (p+q) \notag \\
 &  = \sum_l   \psi_l (p_l-q_l) + \psi_l (p_l + q_l) +4 m_l p_l  + \alpha (p_l -  q_l )^2 + \alpha (p_l + q_l )^2 + 2 m_l^2  \notag \\
 & \qquad \qquad + 2 p \cdot  M  p + 2 q \cdot M q  . \label{e_calc_tilde_H} 
\end{align}
For convenience, we only consider the first summand on the right hand side of~\eqref{e_integral_to_expand}. The second summand can be estimated in the same way. \newline
Due to symmetry of $\tilde H_\alpha (p,q)$ in the $q_l$ variables it holds
\begin{align*}
\int   q_j  \sum_{l \in \Lambda_{\mathrm{tot}}\backslash \Lambda }((p_l + q_l) - m_l)^2 \exp (-   \tilde H_\alpha (p,q)) dp dq  = 0
\end{align*}
Therefore we get by doubling the variable $p$ first and then changing of variables $p = r + \tilde q$ and $\tilde p = r - \tilde q$ that
\begin{align}
& \int  2 p_i  q_j  \sum_{l \in \Lambda_{\mathrm{tot}\backslash \Lambda }}((p_l + q_l) - m_l)^2 \exp (- \tilde H_\alpha (p,q)) dp dq  \notag \\
& \quad = \frac{1}{Z} \int  2 (p_i - \tilde p_i )  q_j  \sum_{l \in \Lambda_{\mathrm{tot}\backslash \Lambda }}((p_l + q_l) - m_l)^2 \notag \\
& \qquad \qquad \times \exp (- \tilde H_\alpha (p,q)- \tilde H_\alpha (\tilde p,q) ) d \tilde p dp dq \notag \\
&  \quad = \frac{1}{Z} \int  4 \tilde q_i   q_j  \sum_{l \in \Lambda_{\mathrm{tot}\backslash \Lambda }}((r_l+\tilde q_l  + q_l) - m_l)^2 \notag \\
& \qquad \qquad \times \exp (-  \tilde{\tilde{H}}_\alpha (r,\tilde q,q) ) d \tilde q dq dr , \label{e_integral_to_expand_finally}
\end{align}
where the Hamiltonian $\tilde{\tilde{H}}_\alpha (r,\tilde q,q)$ is given by
\begin{align*}
&  \tilde{\tilde{H}}_\alpha (r,\tilde q,q) \\
& = \tilde H_\alpha (r + \tilde q , q) + \tilde H_\alpha (r - \tilde q , q) \\
& = H_\alpha (r + \tilde q-q) + H_\alpha (r + \tilde q+q) + H_\alpha (r - \tilde q-q) + H_\alpha (r - \tilde q+q)
\end{align*}
As we have seen in~\eqref{e_calc_tilde_H} from above, the Hamiltonian $\tilde{\tilde{H}}_\alpha (r,\tilde q,q)$ contains no mixed terms in the variables $r, \tilde q$ and $q$. More precisely, $\tilde{\tilde{H}}_\alpha (r,\tilde q,q)$ has three interaction terms i.e. 
\begin{align*}
4  r \cdot M r, \qquad 4 \tilde q \cdot M \tilde q , \qquad \mbox{and} \qquad 2 q \cdot M q.
\end{align*}
So we can rewrite  $\tilde{\tilde{H}}_\alpha (r,\tilde q,q)$ as
\begin{align*}
   \tilde{\tilde{H}}_\alpha (r,\tilde q,q) & = F(r,\tilde q, q) + 4  r \cdot M r + 4  \tilde q \cdot M \tilde q + 2 q \cdot M q,
\end{align*}
where the function $F$ is of the form
\begin{align*}
  F(r, \tilde q, q) = \sum_l \tilde{\psi}_l (r_l, \tilde q_l , q_l)
\end{align*}
for some single-site potentials $\tilde{\psi}_l$ that are symmetric in the variables $\tilde q_l$ and $q_l$. Expanding the term 
\begin{align*}
   \exp(  - 4 \tilde q \cdot M \tilde q  - 2 q \cdot M q)  
\end{align*}
on the right hand side of~\eqref{e_integral_to_expand_finally} yields a sum of terms of the form
\begin{align*}
& -M_{mn}  \int  q_{l_1}^{n_1} \cdots q_{l_k}^{n_k} \tilde q_{\tilde l_1}^{\tilde n_1} \ldots \tilde q_{\tilde l_1}^{\tilde n_1}  ((r_l+\tilde q_l  - q_l) - m_l)^2 \\
& \qquad \times \exp \left(  - \sum_l \tilde{\psi}_l (r_l, \tilde q_l , q_l) -4 r  \cdot (M) r\right) d\tilde q dq dr.
\end{align*}
Because the functions $\tilde{\psi}_l$ are symmetric in the variables $\tilde q_l$ and $q_l$ any term with an odd exponent vanishes. Hence, the exponents $n_1 ,\ldots, n_k, $ and $\tilde n_1 ,\ldots, \tilde n_k, $ are all even. Because $-M_{mn} \geq 0 $ due to the fact that the interaction is ferromagnetic we get 
\begin{align*}
& -M_{mn}  \int  q_{l_1}^{n_1} \cdots q_{l_k}^{n_k} \tilde q_{\tilde l_1}^{\tilde n_1} \ldots \tilde q_{\tilde l_1}^{\tilde n_1}  ((r_l+\tilde q_l  - q_l) - m_l)^2 \\
& \qquad \times \exp \left(  - \sum_l \tilde{\psi}_l (r_l, \tilde q_l , q_l) -4 r  \cdot (M) r\right) d\tilde q dq dr \geq 0.
\end{align*}
All in all, the last inequality yields the desired estimate~\eqref{e_cond_covariances_crucial_estimate} and therefore completes the proof.
\end{proof}

\section{The Logarithmic Sobolev inequality: proof of Theorem~\ref{p_mr_OR}}\label{s_LSI}

This section is devoted to the proof of Theorem~\ref{p_mr_OR}.  We adapt the strategy of Otto \& Reznikoff~\cite[Theorem 3]{OR07} to our situation. 
Recall that compared to Theorem~\ref{p_mr_OR}, we work with weaker assumptions:
\begin{itemize}
\item The single-site potentials $\psi_i$ are only quadratic and not super-quadratic (cf.~\eqref{e_cond_psi} vs.~\eqref{e_single_site_potential_otto}). Also note that in Theorem~\ref{p_OR_original} it is assumed that $M_{ii}=0$, whereas in Theorem~\ref{p_mr_OR} it is assumed that $M_{ii} \geq c >0$ (cf.~\eqref{e_strictly_diag_dominant}). In order to compare both statements it makes sense to think of the single-site potentials in Theorem~\ref{p_mr_OR} as
  \begin{align*}
    \psi_i (x_i) + \frac{1}{2} M_{ii} x_i^2.
  \end{align*}
 \item The interactions $M_{ij}$ decay only algebraically and not exponentially (cf.~\eqref{e_decay_inter_Otto} vs.~\eqref{e_cond_inter_alg_decay_OR}).
\item The correlations are decaying only algebraically and not exponentially (cf. \eqref{e_decay_corr_Otto}~vs.~\eqref{e_cond_alg_decay_OR}).
\end{itemize}
The algebraic decay of interactions and correlations is easy to incorporate in the original argument of~\cite{OR07}, whereas using quadratic and not super-quadratic potentials represents the main technical challenge of the proof. \medskip

The crucial ingredients in the proof of~\cite[Theorem 3]{OR07} are two auxiliary lemmas, namely \cite[Lemma 3 and Lemma 4]{OR07}. A careful analysis of the proof of~\cite{OR07} shows that only this part of the argument is sensitive to weakening the assumptions. Once the analog statements under weaker assumptions (see Lemma~\ref{p_crucial_lemma _1_OR} and Lemma~\ref{p_crucial_lemma _1_OR} below) are verified, the rest of the argument of~\cite[Theorem 3]{OR07} would work the same and is skipped in this article. The remaining part of the argument is based on an recursive application of a general principle, namely the two-scale criterion for LSI (cf.~\cite[Theorem 1]{OR07}), and is therefore not sensitive to changing the assumptions. Hence for the proof of Theorem~\ref{p_mr_OR} it suffices to show that the auxiliary lemmas \cite[Lemma 3 and Lemma 4]{OR07} remain valid under weakening the assumptions. \medskip

Let us turn to the first auxiliary Lemma (cf.~\cite[Lemma~3]{OR07} or Lemma~\ref{p_crucial_lemma _1_OR} from below). It states that the single-site conditional measures satisfy a LSI uniformly in the in the system size and the conditioned spin-values. The argument of \cite[Lemma~3]{OR07} by Otto \& Reznikoff is heavily based on the assumption that the single-site potential $\psi$ is super-quadratic. At this point we provide a new, different, and more elaborated argument showing that the statement of~\cite[Lemma~3]{OR07} remains valid if the single-site potential~$\psi$ is only perturbed quadratic. One could say that the proof of Lemma~\ref{p_crucial_lemma _1_OR} represents the main new ingredient compared to the argument of~\cite{OR07}.

\begin{lemma}[Generalization of~\mbox{\cite[Lemma~3]{OR07}}]\label{p_crucial_lemma _1_OR}

We assume the same conditions as in Theorem~\ref{p_mr_OR}. We consider for an arbitrary subset $S \subset \Lambda_{\mathrm{tot}}$ and site $i \in S$ the single-site conditional measure 
  \begin{align*}
    \bar \mu (dx_i | x^S) := \frac{1}{Z} \exp (- \bar H ( (x^S)) d x_i
  \end{align*}
with Hamiltonian
\begin{align}\label{d_coarse_grained_hamiltonian}
  \bar H (  x^S) = - \log \int \exp(- H(x))  dx^{\Lambda_{\mathrm{tot}} \backslash S}.  
\end{align}
Then the single-site conditional measure $\bar \mu (dx_i | x^S)$ satisfies a LSI with constant $\varrho>0$ (cf.~Definition~\ref{d_LSI}) that is uniform in $\Lambda_{\mathrm{tot}}$, $S$ and the conditioned spins~$x^S$.
\end{lemma}
We state the proof of Lemma~\ref{p_crucial_lemma _1_OR} in Section~\ref{s_crucial_lemma_1_OR}.  \medskip

Let us turn to the second auxiliary Lemma (cf.~\cite[Lemma~4]{OR07} or Lemma~\ref{p_crucial_lemma _2_OR} from below). For some fixed but large enough integer $K$ let us consider the $K$-sublattice $\Lambda_K$ given by
\begin{align}\label{e_def_sublattice}
  \Lambda_K := K \mathbb{Z}^d \cap \Lambda_{\mathrm{tot}}.
\end{align}
Let S an arbitrary subset satisfying $\Lambda_K \subset S \subset \Lambda_{\mathrm{tot}}$. The second auxiliary lemma states that measure on $\Lambda_K$, which is conditioned on the spins in $S \backslash \Lambda_{K}$ and averaged over the spins in $\Lambda_{\mathrm{tot}} \backslash S$, satisfies a LSI with constant $\varrho>0$ uniformly in $S$ and the conditioned spins:
\begin{lemma}[Generalization of~\mbox{\cite[Lemma~4]{OR07}}]\label{p_crucial_lemma _2_OR}
We assume the same conditions as in Theorem~\ref{p_mr_OR}. Let $S$ be an arbitrary set with $\Lambda_K \subset S \subset \Lambda_{\mathrm{tot}}$. Consider the conditional measure
  \begin{align*}
    \bar \mu (dx^{\Lambda_K} | x^{ S \backslash \Lambda_K}) := \frac{1}{Z} \exp (- \bar H ( x^S))  d x^ \Lambda_K
  \end{align*}
with Hamiltonian
\begin{align*}%\label{d_coarse_grained_hamiltonian}
  \bar H (  x^S) = - \log \int \exp(- H(x))  dx^{\Lambda_{\mathrm{tot}} \backslash S}.  
\end{align*}
Then there is some integer $K$ such that the conditional measure $ \bar \mu (dx^{\Lambda_K} | x^{S \backslash \Lambda_K})$ satisfies a LSI with constant $\varrho>0$ (cf.~Definition~\ref{d_LSI}) that is uniform in $\Lambda_{\mathrm{tot}}$, $S$ and the conditioned spins~$x^{S \backslash \Lambda_K}$.
\end{lemma}

\subsection{Proof of Lemma~\ref{p_crucial_lemma _1_OR} and Lemma~\ref{p_crucial_lemma _2_OR}}\label{s_crucial_lemma_1_OR}
Let us first turn to the proof of Lemma~\ref{p_crucial_lemma _1_OR}. For the argument we need the two new ingredients. The first one is the covariance estimate of Proposition~\ref{p_algebraic_decay_correlations} from below. The second one is that the variances of our kind of Gibbs measure are uniformly bounded (see Lemma~\ref{p_est_var_ss} from below). \medskip

Let us now state the covariance estimate of Proposition~\ref{p_algebraic_decay_correlations}.

\begin{proposition}\label{p_algebraic_decay_correlations} 
Let $\Lambda \subset \mathds{Z}^d$ an arbitrary finite subset of the $d$-dimensional lattice $\mathds{Z}^d$. We consider a probability measure $d\mu:= Z^{-1} \exp (-H(x)) \ dx$ on $\mathds{R}^\Lambda$. We assume that
\begin{itemize}
\item the conditional measures $\mu(dx_i | \bar x_i )$, $i \in \Lambda $, satisfy a uniform PI with constant $\varrho_i>0$.
\item the numbers $\kappa_{ij}$, $i \neq j, i,j \in \Lambda$, satisfy
   \begin{equation*}
|\nabla_i \nabla_j H(x)|\leq \kappa_{ij}  < \infty     
   \end{equation*}
uniformly in $x \in \mathds{R}^\Lambda$. Here, $|\cdot|$ denotes the operator norm of a bilinear form. 
\item the numbers $\kappa_{ij}$ decay algebraically in the sense of
  \begin{align}
    \label{e_algeb_decay_of_kappa}
    \kappa_{ij} \lesssim \frac{1}{|i-j|^{d+\alpha} +1} 
  \end{align}
for some $\alpha>0$.
 \item the symmetric matrix $A=(A_{ij})_{N \times N}$ defined by
  \begin{equation*} %\label{e_definition_of_A}
A_{ij} =
\begin{cases}
  \varrho_i, & \mbox{if }\;  i=j , \\
  -\kappa_{ij}, & \mbox{if } \; i< j,
\end{cases}
  \end{equation*}
is strictly diagonally dominant  i.e.~for some $\delta > 0$ it holds for any $i \in \Lambda$ 
\begin{equation}\label{e_strictly_diag_dominant_A}
  \sum_{j \in \Lambda, j \neq i} |A_{ij}| + \delta \le A_{ii}. 
\end{equation}
\end{itemize}
Then for all functions $f=f(x_i)$ and $g=g(x_j)$, $i, j \in \Lambda$, 
   \begin{equation}
     \label{e_covariance_decay_algebraic}
      | \cov_{\mu}(f,g) | \lesssim (A^{-1})_{ij}   \left( \int |\nabla_i f|^2 \ d \mu \right)^{\h} \left( \int |\nabla_j g |^2 \ d \mu \right)^{\h}
   \end{equation}
and for any $i, j \in \Lambda$
\begin{align}
      \label{e_decay_M_inverse}
      |(A^{-1})_{ij}| \lesssim \frac{1}{|i-j|^{d + \tilde \alpha}+1},
    \end{align}
    for some $\tilde \alpha >0$.
  \end{proposition}
For the proof of Proposition~\ref{p_algebraic_decay_correlations} we refer the reader to the article~\cite{Cov_est}.

\begin{proof}[Proof of Lemma~\ref{p_crucial_lemma _1_OR}]

The strategy is to show that the Hamiltonian~$\bar H_i (  x_i )$ of the single-site conditional measure~$  \bar \mu (dx_i | x^S)$ is perturbed strictly- convex in the sense that there exists a splitting
\begin{align}
  \label{e_decom_single_site_hamitlonian}
  \bar H_i (  x_i) = \tilde \psi_i^c (x_i) + \tilde \psi_i^b (x_i)
\end{align}
into the sum of two functions $\tilde \psi_i^c (x_i)$ and $\tilde \psi_i^b (x_i)$ satisfying 
\begin{align}
  \label{e_ssp_perturbed_strictly_convex}
  \tilde (\psi_i^c)'' (x_i) \geq c >0 \quad \mbox{and} \quad |\tilde \psi_i^b (x_i)| \leq C < \infty 
\end{align}
uniformly in $x_i \in \mathbb{R}$, $i\in S$, $\Lambda_{\mathrm{tot}}$ and $S$.\newline
Once~\eqref{e_decom_single_site_hamitlonian} and~\eqref{e_ssp_perturbed_strictly_convex} are validated, the statement of Lemma~\ref{p_crucial_lemma _1_OR} follows simply from a combination of the criterion of Bakry-\'Emery for LSI and the Holley-Stroock perturbation principle (cf.~Appendix~\ref{s_BE_HS} and the proof of~\cite[Lemma 1]{OR07} for details).  \medskip

The aim is to decompose $\bar H_i$ such that~\eqref{e_decom_single_site_hamitlonian} and~\eqref{e_ssp_perturbed_strictly_convex} is satisfied. For that purpose, let us define the auxiliary Hamiltonian $H_{\mathrm{aux}} (x)$, $x \in \mathbb{R}^{\Lambda_{\mathrm{tot}}}$, as
\begin{align}\label{e_def_H_aux}
  H_{\mathrm{aux}} (x) = H(x) - \sum_{j: |j-i| \leq R} \psi_i^b (x_j).
\end{align}
Note that $H_{\mathrm{aux}}$ is strictly convex, if restricted to spins $x_j$ with $|i-j| \leq R$.\newline
For convenience, let us introduce the notation $S^c := \Lambda_{\mathrm{tot}} \backslash S$. The Hamiltonian $\bar H_i$ is then written as
\begin{align*}
  \bar H_i (x_{i}) & \overset{\eqref{d_coarse_grained_hamiltonian}}{=} - \log \int \exp(- H(x)) dx^{S^c} \\
  & = \underbrace{- \log \int \exp(- H_{\mathrm{aux}}(x)) dx^{S^c}}_{=: \tilde \psi_i^c (x_i)}\\
  & \qquad \underbrace{- \log \frac{\int \exp(- H(x)) dx^{S^c}}{\int \exp(- H_{\mathrm{aux}}(x)) dx^{S^c}}}_{=: \tilde \psi_i^b (x_i) }.
\end{align*}
Now, let us check that the functions $\tilde \psi_i^c (x_i)$ and $\tilde \psi_i^b (x_i)$ defined by the last identity satisfy the structural condition~\eqref{e_ssp_perturbed_strictly_convex}. \medskip

Let us consider first the function $\tilde \psi_i^b (x_i)$.  We introduce the auxiliary measure $\mu_{\mathrm{aux}}$ by
\begin{align*}
  \mu_{\mathrm{aux}} (dx^{S^c}) = \frac{1}{Z} \exp \left( - H_{\mathrm{aux}} (x) \right) dx^{S^c}.
\end{align*}
Then it follows from the definition~\eqref{e_def_H_aux} of $H_{\mathrm{aux}}$ that
\begin{align*}
  \left| \tilde \psi_i^b (x_i) \right| & \leq \left| \log \int \exp (- \sum_{j:|j-i| \leq R} \psi_i^b (x_j)) \ \mu_{\mathrm{aux}} (dx^{S^c}) \right| \\
  & \leq \sum_{j:|j-i|\leq R} \| \psi_i^b \|_{\infty} \leq 2(R+1)^d C. 
\end{align*}

It is now left to show that~$\tilde \psi_i^c (x_i)$ is uniformly strictly convex. Direct calculation yields
\begin{align}
  \label{e_second_deriv_tilde_psi_c}
  \frac{d^2}{dx_i^2} \tilde \psi_i^c (x_i) &= \int \frac{d^2}{dx_i^2} H_{\mathrm{aux}} (x) \mu_{\mathrm{aux}} - \var_{\mu_{\mathrm{aux}}} \left( \frac{d}{dx_i} H_{\mathrm{aux}} (x) \right) .
\end{align}
We decompose the measure $\mu_{\mathrm{aux}}$ into
\begin{align*}
&  \mu_{\mathrm{aux}} ( dx^{S^c}) \\
 & \quad = \mu_{\mathrm{aux}} \left( (dx_j)_{j \in S^c, |j-i| \leq R} \ | \ (x_j)_{j \in S^c, |j-i|> R}  \right ) \bar \mu_{\mathrm{aux}} ( (dx_j)_{j \in S^c, |j-i|> R}).
\end{align*}
Here, $\mu_{\mathrm{aux}} \left( (dx_j)_{j \in S^c, |j-i| \leq R} \ | \ (x_j)_{j \in S^c, |j-i | > R}  \right )$ denotes the conditional measure given by
\begin{align*}
 & \mu_{\mathrm{aux}} \left(  (dx_j)_{j \in S^c, |j-i| \leq R} \ | \ (x_j)_{j \in S^c, |j-i|> R}  \right ) \\
 & \qquad = \frac{1}{Z} \exp\left( - H_{\mathrm{aux}} (x) \right) \   \otimes_{\substack{j \in S^c, \\ |j-i| \leq R}}  dx_j ,
\end{align*}
whereas $\bar \mu_{\mathrm{aux}} ( (dx_j)_{j \in S^c, |j-i|> R})$ denotes the marginal measure given by
\begin{align*}
&  \bar \mu_{\mathrm{aux}} ( (dx_j)_{j \in S^c, |j-i|> R }) \\ & \qquad = \frac{1}{Z} \left( \int \exp\left( - H_{\mathrm{aux}} (x) \right) \   \otimes_{\substack{l \in  S^c, \\ |l-i| \leq R}}  dx_l \right) \ \ \otimes_{\substack{j \in S^c, |j-i|> R }}  dx_j  .
\end{align*}
For convenience, we write $\mu_{\mathrm{aux},c}$ instead of the conditional measure $\mu_{\mathrm{aux}} \left(  (dx_j)_{j \in S^c, |j-i| \leq R} \ | \ (x_j)_{j \in S^c, |j-i|> R}  \right )$.
Applying the decomposition to~\eqref{e_second_deriv_tilde_psi_c} yields
\begin{align}
  \frac{d^2}{dx_i^2} \tilde \psi_i^c (x_i) &= \int \left(  \int  \frac{d^2}{dx_i^2} H_{\mathrm{aux}} (x) \mu_{\mathrm{aux},c}   - \var_{\mu_{\mathrm{aux},c}} \left( \frac{d}{dx_i} H_{\mathrm{aux}} (x) \right)  \right)  \  \bar \mu_{\mathrm{aux}} \notag \\
& \qquad  - \var_{\bar \mu_{\mathrm{aux}}} \left( \int \frac{d}{dx_i} H_{\mathrm{aux}} (x) \mu_{\mathrm{aux},c}\right). \label{e_decomp_desintegration}
\end{align}
The first term on the right hand side of the last identity is controlled easily. Note that the Hamiltonian $H_{\mathrm{aux}}$ is strictly-convex, if restricted to spins $x_j$ with $|j-i|\leq R$. So it follows from a standard argument based on the Brascamp lieb inequality that (for details see for example~\cite[Chapter 3]{Dizdar})
\begin{align*}
   \int  \frac{d^2}{dx_i^2} H_{\mathrm{aux}} (x) \mu_{\mathrm{aux},c}   - \var_{\mu_{\mathrm{aux},c}} \left( \frac{d}{dx_i} H_{\mathrm{aux}} (x) \right)  \geq c >0
\end{align*}
uniformly in $R$ and therefore also
\begin{align*}
  \int \left(  \int  \frac{d^2}{dx_i^2} H_{\mathrm{aux}} (x) \mu_{\mathrm{aux},c}   - \var_{\mu_{\mathrm{aux},c}} \left( \frac{d}{dx_i} H_{\mathrm{aux}} (x) \right)  \right)  \  \bar \mu_{\mathrm{aux}} \geq c >0 
\end{align*}
uniformly in $R$. \medskip

Let us now turn to the second term in~\eqref{e_decomp_desintegration}. Straightforward calculation yields
\begin{align*}
  \frac{d}{dx_i} H_{\mathrm{aux}} (x) & = (\psi_i^c)' (x_i) + M_{ii} x_i + s_i + \frac{1}{2}  \sum_{j \in \Lambda_{\mathrm{aux}}} M_{ij} x_j.
\end{align*}
Because the measures $\bar \mu_{\mathrm{aux}}$ and $\mu_{\mathrm{aux},c}$ live on a subset of $S^c$, $i \in S$, and the variance is invariant under adding constants, we have 
\begin{align}
&  \var_{\bar \mu_{\mathrm{aux}}} \left( \int \frac{d}{dx_i} H_{\mathrm{aux}} (x) \mu_{\mathrm{aux},c}\right) = \var_{\bar \mu_{\mathrm{aux}}} \left( \frac{1}{2} \int \sum_{j \in S^c} M_{ij} x_j \  \mu_{\mathrm{aux},c}\right) \notag \\
  & \quad = \frac{1}{4}\var_{\bar \mu_{\mathrm{aux}}} \left( \sum_{\substack{j \in S^c, \\ |j-i| > R}} M_{ij} x_j \right)  \\
& \qquad + \frac{1}{4}\var_{\bar \mu_{\mathrm{aux}}} \left( \int \sum_{\substack{j \in S^c , \\ |j-i| \leq R}} M_{ij} x_j \  \mu_{\mathrm{aux},c}\right). \label{e_decomp_var_crucail_lemma_1_OR}
\end{align}
The first summand on the right hand side of the last identity is estimated in a straightforward manner i.e.
\begin{align*}
&  \var_{\bar \mu_{\mathrm{aux}}} \left( \sum_{\substack{j \in S^c, \\ |j-i| > R}} M_{ij} x_j \right) \\
& = \var_{ \mu_{\mathrm{aux}}} \left( \sum_{\substack{j \in S^c, \\ |j-i| > R}} M_{ij} x_j \right) \\
 & =  \sum_{\substack{j \in S^c, \\ |j-i| > R}} M_{ij} \sum_{\substack{l \in S^c, \\ |l-i| > R}} M_{il} \cov_{\mu_{\mathrm{aux}}} (x_j,x_l) \\
& \leq \sum_{\substack{j \in S^c, \\ |j-i| > R}} \sum_{\substack{l \in S^c, \\ |l-i| > R}} M_{ij}  M_{il} \left( \var_{\mu_{\mathrm{aux}}} (x_j) \right)^{\frac{1}{2}}  \left( \var_{\mu_{\mathrm{aux}}} (x_l) \right)^{\frac{1}{2}}\\
& \overset{\eqref{e_est_ss_var}}{\leq}  C \sum_{\substack{j \in S^c, \\ |j-i| > R}} \sum_{\substack{l \in S^c, \\ |l-i| > R}} M_{ij}  M_{il} \\
& \overset{~\eqref{e_cond_inter_alg_decay_OR}}{\leq}  C \sum_{\substack{j \in S^c, \\ |j-i| > R}} \sum_{\substack{l \in S^c, \\ |l-i| > R}} \frac{1}{|i-j|^{d + \alpha} +1 } \ \frac{1}{|i-l|^{d + \alpha} +1 }   \\
& \leq C \frac{1}{R^{\frac{\alpha}{2}}}.
\end{align*}
Here we have used one of the new ingredients, namely the uniform estimate~\eqref{e_est_ss_var} stated in Lemma~\ref{p_est_var_ss} from below. Note that Lemma~\ref{p_est_var_ss} also applies to the measure~$\mu_{\mathrm{aux}}$ because~$\mu_{\mathrm{aux}}$ satisfies the same structural assumptions as the measure~$\mu_{\Lambda}$. \newline
Let us consider now the second summand on the right hand side of~\eqref{e_decomp_var_crucail_lemma_1_OR}. By doubling the variables we get
\begin{align*}
 & \var_{\bar \mu_{\mathrm{aux}}} \left( \int \sum_{\substack{j \in S^c , \\ |j-i| \leq R}} M_{ij} x_j \  \mu_{\mathrm{aux},c}\right) \\
 & \quad = \int \Big( \int \sum_{\substack{j \in S^c , \\ |j-i| \leq R}} M_{ij} x_j \  \mu_{\mathrm{aux},c} (dx| y) \\
& \quad \qquad - \int \sum_{\substack{j \in S^c , \\ |j-i| \leq R}} M_{ij} x_j \  \mu_{\mathrm{aux},c}  (dx| \bar y )\Big)^2 \bar \mu_{\mathrm{aux}} (dy)  \bar \mu_{\mathrm{aux}} (d \bar y)
\end{align*}
By interpolation we have
\begin{align*}
&   \int \sum_{\substack{j \in S^c , \\ |j-i| \leq R}} M_{ij} x_j \  \mu_{\mathrm{aux},c} (dx| y)- \int \sum_{\substack{j \in S^c , \\ |j-i| \leq R}} M_{ij} x_j \  \mu_{\mathrm{aux},c}  (dx| \bar y )  \\ 
& \quad = \int_0^1 \frac{d}{dt} \int \sum_{\substack{j \in S^c , \\ |j-i| \leq R}} M_{ij} x_j \  \mu_{\mathrm{aux},c} (dx| ty + (1-t) \bar y) \ dt \\
& \quad = \int_0^1 \cov_{\mu_{\mathrm{aux},c} (dx| ty + (1-t) \bar y)} \left( \sum_{\substack{j \in S^c , \\ |j-i| \leq R}} M_{ij} x_j, \sum_{\substack{k,l \in S^c , \\ |k-i| \leq R \\ |l-i| \geq R }} x_k M_{kl} (\bar y_l - y_l) \right) \ dt \\
& \quad = \int_0^1 \sum_{\substack{j \in S^c , \\ |j-i| \leq R}} M_{ij} M_{kl} (\bar y_l - y_l) \sum_{\substack{k,l \in S^c , \\ |k-i| \leq R \\ |l-i| \geq R }}  \cov_{\mu_{\mathrm{aux},c} (dx| ty + (1-t) \bar y)} \left(  x_j,  x_k  \right) \ dt.
\end{align*}
Without loss of generality we may assume that the interaction is ferromagnetic i.e.~$M_{kl} \leq 0$ for all $k \neq l$ (else use $M_{kl}\leq|M_{kl}|$ and Lemma~\ref{p:attractive_interact_dominates}). Note that the measure $\mu_{\mathrm{aux},c}$ has strictly convex single-site potentials. Therefore the single-site conditional measures $\mu (dx_1| x)$ satisfy a LSI with constant $\frac{1}{2}M_{ii}$ by the Bakry-\'Emery criterion (see Theorem~\ref{local:thm:BakryEmery}). Because the interaction is strictly-diagonally dominant in the sense of~\eqref{e_strictly_diag_dominant}, an application of Proposition~\ref{p_algebraic_decay_correlations} yields that the covariance can be estimated as
  \begin{align*}
    \cov_{\mu_{\mathrm{aux},c} (dx| ty + (1-t) \bar y)} \left(  x_j,  x_k  \right) \leq (M^{-1})_{jk},
  \end{align*}
where the matrix $M$ is given by the elements
\begin{align*}
  M_{ln} \quad \mbox{for} \quad  l , n \in S^c , |l-i| \leq R, |k-i| \leq R  \quad \mbox{or} \quad l=n=i. 
\end{align*}
We want to note that by an simple standard result (see for example~\cite[Lemma 5]{OR07}) or~\cite[Lemma 4.3]{MN}) it holds $(M^{-1})_{kl} \geq 0$ for all $k,l$. Using this information, we get by an application of Jensen's inequality that
  \begin{align*}
 & \var_{\bar \mu_{\mathrm{aux}}} \left( \int \sum_{\substack{j \in S^c , \\ |j-i| \leq R}} M_{ij} x_j \  \mu_{\mathrm{aux},c}\right) \\
& \quad \leq C  \int_0^1 \sum_{\substack{j \in S^c , \\ |j-i| \leq R}}  \sum_{\substack{k,l \in S^c , \\ |k-i| \leq R \\ |l-i| \geq R }}  M_{ij} M_{kl}   (M^{-1})_{jk} \int (\bar y_l - y_l)^2 \ \bar \mu_{\mathrm{aux}} (dy) \bar \mu_{\mathrm{aux}} (d \bar y)\ dt \ \\
& \quad \leq C   \sum_{\substack{j \in S^c , \\ |j-i| \leq R}}  \sum_{\substack{k,l \in S^c , \\ |k-i| \leq R \\ |l-i| \geq R }}  M_{ij} M_{kl}   (M^{-1})_{jk} \var_{\bar \mu_{\mathrm{aux}}}( y_l)\\
& \quad \leq C  \sum_{\substack{j \in S^c , \\ |j-i| \leq R}} \sum_{\substack{k,l \in S^c , \\ |k-i| \leq R \\ |l-i| \geq R }}  M_{ij} M_{kl}  (M^{-1})_{jk} \var_{ \mu_{\mathrm{aux}}}( y_l)\\
& \quad \overset{\eqref{e_est_ss_var}}{\leq} C   \sum_{\substack{j \in S^c , \\ |j-i| \leq R}}  \sum_{\substack{k,l \in S^c , \\ |k-i| \leq R \\ |l-i| \geq R }}  M_{ij} M_{kl}  (M^{-1})_{jk} \\
& \quad \overset{~\eqref{e_cond_inter_alg_decay_OR}, \eqref{e_decay_M_inverse}}{\leq} C   \sum_{\substack{j \in S^c , \\ |j-i| \leq R}}  \sum_{\substack{k,l \in S^c , \\ |k-i| \leq R \\ |l-i| \geq R }} \frac{1}{|i-j|^{d+\alpha} +1} \ \frac{1}{|k-l|^{d+\alpha} +1} \frac{1}{|j-k|^{d+\tilde \alpha} +1} \\
& \quad \leq \frac{C}{R^{\frac{\tilde \alpha}{{2}}}} 
\end{align*}
Note that here we also used the second ingredient, namely the covariance estimates~\eqref{e_covariance_decay_algebraic} and~\eqref{e_decay_M_inverse}. 
Hence, both terms on the right hand side of~\eqref{e_decomp_var_crucail_lemma_1_OR} are arbitrarily small, if we choose $R$ big enough. Overall this leads to the desired statement (cf.~\eqref{e_decomp_desintegration} ff.)
\begin{align*}
  \frac{d^2}{dx_i^2} \tilde \psi_i^c (x_i) \geq c >0,
\end{align*}
which completes the argument.
\end{proof}

In the proof of Lemma~\ref{p_crucial_lemma _1_OR}, we needed the following auxiliary statement.
\begin{lemma}\label{p_est_var_ss}
  Under the same assumptions as in Lemma~\ref{p_crucial_lemma _1_OR}, it holds that for all $i \in \Lambda$
  \begin{align}\label{e_est_ss_var}
    \var_{\mu_\Lambda} (x_i) \leq C,
  \end{align}
where the bound is uniform in $\Lambda$ and only depends on the constants appearing in~\eqref{e_cond_psi} and in~\eqref{e_strictly_diag_dominant}. 
\end{lemma}
The proof of Lemma~\ref{p_est_var_ss} is a simple and straightforward application of a exponential moment bound due to Robin Nittka.
\begin{lemma}[\mbox{\cite[Lemma~4.3]{MN}}]\label{lem:moments}
  We assume that the formal Hamiltonian $H:\mathds{R}^{\mathds{Z}^d} \to \mathds{R} $ given by~\eqref{e_d_Hamiltonian} satisfies the Assumptions~\eqref{e_cond_psi}~-~\eqref{e_strictly_diag_dominant}.\newline
Additionally, we assume that for all $ i \in \mathds{Z}^d$ the convex part $\psi_i^c$ of the single-site potentials $\psi_i$ has a global minimum in $x_i=0$. \newline
Let $\delta >0$ be given by~\eqref{e_strictly_diag_dominant}. Then for every $0 \le a \le \frac{\delta}{2}$ and any subset $\Lambda \subset \mathds{Z}^d$ it holds
        \begin{equation} \label{e:exponential_moment}
          \E_{\mu_{\Lambda}} \bigl[\e^{a p_i^2}\bigr] \lesssim 1.
        \end{equation}
In particular, for any $k \in \mathds{N}_0$ this yields 
\begin{equation} \label{e:arbitrary_moment}
  \E_{\mu_{\Lambda}}[p_i^{2k}] \lesssim k!.
\end{equation}
\end{lemma}
The statement of Lemma~\ref{lem:moments} is a slight improvement of~\cite[Section~3]{BHK82}, because the assumptions are slightly weaker compared to~\cite {BHK82}. More precisely, $\psi_i''$ may change sign outside every compact set and there is no condition on the signs of the interaction. Even if \cite[Lemma~4.3]{MN} is formulated in~\cite{MN} for systems on an one-dimensional lattice, a simple analysis of the proof shows that the statement is also true on lattices of any dimension. 
\begin{proof}[Proof of Lemma~\ref{p_est_var_ss}]
By doubling the variables we get 
\begin{equation*}
  \var_{\mu_{\Lambda}}(x_i) = \frac{1}{2} \int \int (x_i -y_i)^2  \mu_{\Lambda}(dx) \mu_{\Lambda} (dy).
\end{equation*}
By the change of coordinates $x_k= q_k + p_k$ and $y_k= q_k -p_k$ for all $k \in \Lambda$, the last identity yields by using the definition~\eqref{e_d_Gibbs_measure} of the finite-volume Gibbs measure $\mu_{\Lambda}$ that
\begin{align*}
  \var_{\mu_{\Lambda}}(x_i) &= C \int \int p_i^2  \  \underbrace{\frac{e^{-H(q^{\Lambda}+p^{\Lambda}, x^{\mathds{Z}\backslash \Lambda}) - H(q^{\Lambda}-p^{\Lambda}, x^{\mathds{Z}\backslash \Lambda})}}{\int e^{-H(q^{\Lambda}+p^{\Lambda}, x^{\mathds{Z}\backslash \Lambda})  -H(q^{\Lambda}-p^{\Lambda}, x^{\mathds{Z}\backslash \Lambda})}  \dx p^{\Lambda} \dx q^{\Lambda} }   \dx p^{\Lambda} \dx q^{\Lambda} }_{=: d \tilde \mu_{\Lambda} (q^{\Lambda},p^{\Lambda})} . 
\end{align*}
By conditioning on the values $q^{\Lambda}$ it directly follows from the definition~\eqref{e_d_Hamiltonian} of $H$ that
\begin{equation} \label{e:repre_covariance}
  \var_{\mu_{\Lambda}}(x_i) = C
 \mathds{E}_{\tilde \mu_{\Lambda}} \left[ \mathds{E}_{\mu_{\Lambda,q}} \left[ p_i^2 \right] \right].
\end{equation}
Here, the conditional measure $\mu_{\Lambda,q}$ is given by the density
\begin{equation}
  \label{eq:def_mu_q}
	\dx\mu_{\Lambda,q}(p^\Lambda) \coloneqq \frac{1}{Z_{\mu_{\Lambda,q}}} \e^{-\sum_{k \in \Lambda} \psi_{k,q}(p_k) - \sum_{k,l \in \Lambda} M_{kl} p_k p_l} \dx p^\Lambda  
\end{equation}
with single-site potentials $\psi_{k,q} \coloneqq \psi_{k,q}^c + \psi_{k,q}^b$ defined by 
\begin{align*}
   \psi_{k,q}^c(p_k) & \coloneqq \psi_k^c(q_k + p_k) + \psi_k^c(q_k - p_k) \qquad \mbox{and} \\
  \psi_{k,q}^b(p_k) & \coloneqq \psi_k^b(q_k + p_k) + \psi_k^b(q_k - p_k).  
\end{align*}
Because of symmetry in the variable $p_k$,the convex part of the single-site potential $\psi_{k,q}^c(p_k)$ has a global minimum at $p_k=0$ for any $k$. Therefore, an application of Lemma~\ref{lem:moments} yields the desired statement. \end{proof}

Let us turn to the verification of Lemma~\ref{p_crucial_lemma _2_OR}. We also need an auxiliary statement, namely Lemma~\ref{p_aux_lemma_2_OR} from below.
It is a generalization of~\cite[Lemma 2]{OR07} and states that the interactions of the Hamiltonian $\bar H ((x_i)_{i \in S})$ given by~\eqref{d_coarse_grained_hamiltonian} decay sufficiently fast.
\begin{lemma} [Generalization of~\mbox{\cite[Lemma 2]{OR07}}]\label{p_aux_lemma_2_OR}
In the same situation as in Lemma~\ref{p_crucial_lemma _1_OR}, the interactions of $\bar H (x^S)$ decay algebraically i.e. there are constants $0, \varepsilon, C < \infty$  such that
\begin{align*}
  \left|\frac{d}{dx_i} \frac{d}{dx_j}  \bar H \right| \leq C \frac{1}{|i-j|^{d+\bar \varepsilon}+1}
\end{align*} 
uniformly in $i,j \in S$.
\end{lemma}
\begin{proof}[Proof of Lemma~\ref{p_aux_lemma_2_OR}]
  Direct calculation as in~\cite[Lemma 2]{OR07} shows that 
\begin{align*}
  \frac{d}{dx_i} \frac{d}{dx_j}  \bar H = -M_{ij} - \sum_{k \in \Lambda_{\mathrm{tot}} \backslash S}  \sum_{l \in \Lambda_{\mathrm{tot}} \backslash S} M_{ik} \ M_{jl} \  \cov_{\Lambda_{\mathrm{tot} \backslash S}} (x_k, x_l).
\end{align*} 
The last identity immediately yields the estimate (cf.~\cite[(52)]{OR07})
\begin{align*}
  \left|\frac{d}{dx_i} \frac{d}{dx_j}  \bar H \right| \leq |M_{ij}| + \sum_{k \in \Lambda_{\mathrm{tot}} \backslash S}  \sum_{l \in \Lambda_{\mathrm{tot}} \backslash S} |M_{ik}| \ |M_{jl}| \ | \cov_{\Lambda_{\mathrm{tot} \backslash S}} (x_k, x_l)|
\end{align*} 
 Using the decay of interactions~\eqref{e_decay_inter_Otto} and the decay of correlations~\eqref{e_decay_corr_Otto} we get
\begin{align*}
 & \left|\frac{d}{dx_i} \frac{d}{dx_j}  \bar H \right| \\
 & \quad \leq \frac{C}{|i-j|^{d+ \alpha}} + C \sum_{k \in \Lambda_{\mathrm{tot}} \backslash S}  \sum_{l \in \Lambda_{\mathrm{tot}} \backslash S} \frac{1}{|i-k|^{d+ \alpha}}    \frac{1}{|j-l|^{d+ \alpha}}          \frac{1}{|k-l|^{d+ \alpha}} 
\end{align*} 
Now we use the same kind of argument as used in in the proof of Proposition~\ref{p_algebraic_decay_correlations} to estimate the term $T_k$. This means that for any multi-indexes $i,k,l,j \in \Lambda_{\mathrm{tot}}$ it holds either  
\begin{align*}
%  \label{e_cond_dist_jumps}
  |i-k|\geq \frac{C}{3} |i-j|, \quad  |j-l|\geq \frac{C}{3} |i-j|, \quad \mbox{or} \quad |k-l|\geq \frac{C}{3} |i-j|.    
\end{align*}
Therefore,  we have
\begin{align*}
& \sum_{k \in \Lambda_{\mathrm{tot}} \backslash S}  \sum_{l \in \Lambda_{\mathrm{tot}} \backslash S} \frac{1}{|i-k|^{d+ \alpha}}    \frac{1}{|j-l|^{d+ \alpha}}          \frac{1}{|k-l|^{d+ \alpha}}   \\
& \leq \sum_{ \substack{k \in \Lambda_{\mathrm{tot}} \backslash S, \\ l \in \Lambda_{\mathrm{tot}} \backslash S, \\ |i-k|\geq \frac{1}{3} |i-j|}} \frac{1}{|i-k|^{d + \alpha}}    \frac{1}{|j-l|^{d + \alpha}}          \frac{1}{|k-l|^{d + \alpha}} \\
& \qquad + \sum_{ \substack{k \in \Lambda_{\mathrm{tot}} \backslash S, \\ l \in \Lambda_{\mathrm{tot}} \backslash S, \\ |j-l|\geq \frac{1}{3} |i-j|}} \ldots + \sum_{ \substack{k \in \Lambda_{\mathrm{tot}} \backslash S, \\ l \in \Lambda_{\mathrm{tot}} \backslash S, \\ |k-l|\geq \frac{1}{3} |i-j|}} \ldots \\
& \leq C \ \frac{1}{|i-j|^{d + \alpha}}, 
\end{align*} 
which yields the desired statement of Lemma~\ref{p_aux_lemma_2_OR}.
\end{proof}
As in the proof of~\cite[Lemma~4]{OR07} we verify Lemma~\ref{p_crucial_lemma _2_OR} by an application of the Otto-Reznikoff criterion for LSI i.e.
\begin{theorem}[Otto-Reznikoff criterion for LSI,~{\mbox{\cite[Theorem 1]{OR07}}}] \label{p_otto_reznikoff}
  Let $d\mu:= Z^{-1} \exp (-H(x)) \ dx$ be a probability measure on a direct product of Euclidean spaces $X= X_1 \times \cdots \times X_N$. We assume that
\begin{itemize}
 \item the conditional measures $\mu(dx_i | \bar x_i )$, $1\leq i \leq N$, satisfy a uniform LSI($\varrho_i $).
 \item the numbers $\kappa_{ij}$, $1 \leq i \neq j \leq N$, satisfy
   \begin{equation*}
 |\nabla_i \nabla_j H(x)|\leq \kappa_{ij} < \infty     
   \end{equation*}
uniformly in $x \in X$. Here, $|\cdot|$ denotes the operator norm of a bilinear form. 
\item the symmetric matrix $A=(A_{ij})_{N \times N}$ defined by
  \begin{equation*}
A_{ij} =
\begin{cases}
  \varrho_i, & \mbox{if } \; i=j , \\
  -\kappa_{ij}, & \mbox{if } \; i < j,
\end{cases}
  \end{equation*}
satisfies in the sense of quadratic forms 
\begin{equation}\label{e_cond_OR}
   A \geq \varrho \Id \qquad \mbox{for a constant } \varrho>0.
\end{equation}
\end{itemize}
Then $\mu$ satisfies LSI($\varrho$).
\end{theorem}
\begin{proof}[Proof of Lemma~\ref{p_crucial_lemma _2_OR}]
  We want to apply Theorem~\ref{p_otto_reznikoff}. By an application of Lemma~\ref{p_crucial_lemma _1_OR}, we know that the single-site measures conditional measures $\bar \mu (dx_i |x^{\Lambda_K} x^{ S \backslash,  \Lambda_K})$, $i \in \Lambda_k$ satisfy a LSI with uniform constant $\varrho>0$. \newline
For the mixed derivatives of the Hamiltonian, we have according to Lemma~\ref{p_aux_lemma_2_OR}
\begin{align*}
  \left|\frac{d}{dx_i} \frac{d}{dx_j}  \bar H \right| \leq C \frac{1}{|i-j|^{d+\bar \varepsilon}+1}.
\end{align*} 
Hence, in order to apply Theorem 1 we have to consider the symmetric matrix $A = (A_{ij})_{i,j \in \Lambda_K}$ with
  \begin{align*}
      & A_{ii}= \varrho, \\
      & A_{ij} = - C \frac{1}{|i-j|^{d+\bar \varepsilon} +1}, \qquad \mbox{for } i\neq j.
  \end{align*}
We will argue that A is strict positive-definite if we choose the integer $K$ large enough. 
We have
\begin{align*}
  \sum_{i,j \in \Lambda_K} x_i A_{ij} x_j = \sum_{i \in \Lambda_K} \varrho x_i^2 + \sum_{i,j \in \Lambda_K, \ i \neq j} x_i A_{ij} x_j.
\end{align*}
Let us estimate the second term of the right hand side. We have
\begin{align*}
 | \sum_{i,j \in \Lambda_K, \ i \neq j} x_i A_{ij} x_j | & \leq \frac{1}{2} \sum_{i \in \Lambda_K} \sum_{j \in \Lambda_K, \ i \neq j} |A_{ij}| x_i^2 + \frac{1}{2} \sum_{j \in \Lambda_K} \sum_{i \in \Lambda_K, \ i \neq j} |A_{ij}| x_j^2\\
& \leq C \sum_{i \in \Lambda_K} \sum_{j \in \Lambda_K, \ i \neq j} \frac{1}{|i-j|^{d + \bar \varepsilon}} x_i^2 \\
 & \leq \frac{C}{K^{\frac{\bar \varepsilon}{2}}} \sum_{i \in \Lambda_K} x_i^2 \sum_{j \in \Lambda_K, \ i \neq j} \frac{1}{|i-j|^{d + \frac{\bar \varepsilon}{2}}} 
\end{align*}
where the last inequality holds if we choose $K$ large enough. So we get overall that 
\begin{align*}
  \sum_{i,j \in \Lambda_K} x_i A_{ij} x_j \geq \frac{\varrho}{2} \sum_{i \in \Lambda_K} x_i^2 >0 ,
\end{align*}
 which yields the desired statement of Lemma~\ref{p_crucial_lemma _2_OR} by an application of Theorem~\ref{p_otto_reznikoff}.

\end{proof}

\appendix

\section{Uniqueness of the infinite-volume Gibbs measure: proof of Theorem~\ref{p_unique_Gibbs}}\label{s_decay_and_uniqueness}

The proof of Theorem~\ref{p_unique_Gibbs} is straightforward and only needs four ingredients:
\begin{itemize}
\item A sufficient decay of interactions (cf.~\eqref{e_decay_inter_Otto}).
\item A sufficient decay of correlations (cf.~\eqref{e_decay_corr_Otto}).
\item The uniform PI for the finite-volume Gibbs measures $\mu_{\Lambda}$. This is provided by Theorem~\ref{p_mr_OR} and the fact that the LSI yields a PI with the same constant.
\item The fact that the variances of the infinite-volume Gibbs measure $\mu$ are uniformly bounded (cf.~\eqref{e_sup_moment}).
\end{itemize}

\begin{proof}[Proof of Theorem~\ref{p_unique_Gibbs}]
  Let us assume that there are two infinite-volume Gibbs measures $\mu$ and $\tilde \mu$. It suffices to show that for a function $f$ with bounded support 
\begin{equation}\label{e_desired_statement}
 \left| \int f \mu - \int f \tilde \mu  \right| = 0.
\end{equation}
Let $B_R$ denote a ball with radius $R$ and center in the root of the lattice $\mathbb{Z}^d$. We decompose the measures $\mu$ and $\tilde \mu$ w.r.t.~$B_R$ into 
\begin{align*}
&\mu(dx^{B_R},d\omega^{\mathds{Z}^d \backslash B_R}) =  \mu(dx^{B_R}| \omega^{\mathds{Z}^d \backslash B_R}) \bar \mu (d \omega^{\mathds{Z}^d \backslash B_R})  \qquad \mbox{and} \\
& \tilde \mu(dx^{B_R},d \tilde \omega^{\mathds{Z}^d \backslash B_R}) =  \tilde \mu(dx^{B_R}  | \tilde \omega^{\mathds{Z}^d \backslash B_R}) \tilde{\bar \mu} (d \tilde \omega^{\mathds{Z}^d \backslash B_R}),
\end{align*}
where $\mu(dx^B_R| \omega^{\mathds{Z}^d \backslash B_R} )$ and $\tilde \mu(dx^B_R |\tilde \omega^{\mathds{Z}^d \backslash B_R})$ denote the conditional measures and $\bar \mu (d \omega^{\mathds{Z}^d \backslash B_R})$ and $\tilde{\bar \mu} (d \tilde \omega^{\mathds{Z}^d \backslash B_R})$ denote the corresponding marginals. For convenience we will write $x$ and $\omega$ instead of $x^{B_R}$ and $\omega^{\mathds{Z}^d \backslash B_R}$. \smallskip

From the DLR-equations it follows that $$\mu(dx|\omega) = \tilde \mu(dx| \omega)= \mu_\Lambda(dx|\omega),$$
where $\mu_\Lambda(dx)$ denotes the finite-volume Gibbs measure associated to the tempered state $\omega$ given by~\eqref{e_d_Gibbs_measure}. Hence, we get
\begin{align*}
  \left| \int f \mu - \int f \tilde \mu  \right| & =  \left| \int \int f \mu(dx|\omega) \bar \mu (d \omega) - \int \int f \mu(dx| \tilde \omega) \bar{\tilde\mu} (d\tilde \omega)  \right| \\
  & = \left| \int \int \left(\int  f \mu_\Lambda(dx|\omega)  - \int f \mu_\Lambda(dx| \tilde \omega)\right) \bar \mu (d \omega)  \bar{\tilde\mu} (d\tilde \omega)  \right| \\
 & \leq  \int \int \left| \left(\int  f \mu(dx|\omega)  - \int f \mu(dx| \tilde \omega)\right) \right| \bar \mu (d \omega)  \bar{\tilde\mu} (d\tilde \omega)   .
\end{align*}
The statement of Theorem~\ref{p_unique_Gibbs} follows, once we have shown that
\begin{align}\label{e_decay_dist_supp}
& \int \int  \left| \left(\int  f \mu(dx|\omega)  - \int f \mu(dx| \tilde \omega)\right) \right| \bar \mu (d \omega) \bar{\tilde\mu} (d\tilde \omega) \\
&  \qquad  \leq C(f) \ \frac{1}{|\dist ( \supp f , \mathds{Z}^d \backslash B_R)|^{\varepsilon}} \notag
\end{align}
for some $\varepsilon >0$. The reason is that by choosing the size of the ball $R \to\infty $ the estimate~\eqref{e_decay_dist_supp} yields 
\begin{align*}
  \left| \int f \mu - \int f \tilde \mu \right| \to 0.
\end{align*}
\medskip

Let us verify~\eqref{e_decay_dist_supp}.  We define for $t \in [0,1]$ the measures $\mu (dx | t \omega + (1-t) \tilde \omega)$ that interpolate between $\mu(dx| \tilde \omega)$ and $\mu(dx| \omega)$. Hence we get
\begin{align*}
  \int f &\mu (dx | \omega) -   \int f \mu (dx | \tilde \omega)  = \int_0^1 \frac{d}{dt}  \mu (dx | t \omega + (1-t) \tilde \omega) \\
  & =\int_0^1 \cov_{\mu(dx | t \omega + (1-t) \tilde \omega)} \left( f, \sum_{i \in B_R, \ j \notin B_R}  M_{ij} x_i (\tilde \omega - \omega) \right) .
\end{align*}
So, we have to estimate the covariance term on the right hand side of the last equation. For convenience, we write 
\begin{align*}
   &  x= (x_i) , \qquad \mbox{for } i \in \supp (f), \\
   & z = (z_l) , \qquad \mbox{for } l \in B_R \backslash \supp (f),  \\
   & \omega= (\omega_j) , \qquad \mbox{for } j \notin B_R \\
   & \omega_t = t \omega + (1-t) \tilde \omega , \quad \mbox{and} \\
   & \Delta \omega = \tilde \omega - \omega.
\end{align*}
With this notation, we decompose the measure $\mu(dx dz | \omega_t)$ according to
\begin{align*}
  \mu (dx dz |  \omega_t ) =    \mu (dz |  \omega_t )   \bar \mu (dx |  \omega_t ) .
\end{align*}
This decomposition of $\mu (dx dz |  \omega_t )$ yields the following decomposition of the covariance, namely
\begin{align*}
&  \cov_{\mu(dx dz | \omega_t)} \left( f, \sum_{i \in \supp (f), j \notin B_R} M_{ij} x_i \Delta \omega_j +  \sum_{l \in B_R \backslash \supp (f), j \notin B_R} M_{ij} z_l \Delta\omega_j \right) \\
  & \qquad =  \underbrace{\cov_{\bar \mu(dx| \omega_t)}  \left( f, \sum_{i , j } M_{ij} x_i \Delta \omega_j \right)}_{=: T_1} \\
  & \qquad \qquad + \underbrace{ \cov_{ \bar \mu(dx|  \omega_t)}  \left( f,  \int \sum_{l , j } M_{ij} z_l \Delta \omega_j  \mu (dz | x, \omega_t) \right) }_{=: T_2}.
\end{align*}
For, convenience here and from now on, the indexes $i$, $l$, and $j$ are always belonging to the set
\begin{align*}
  i \in \supp (f) \qquad l \in B_R \backslash \supp (f), \quad \mbox{and} \quad j \in B_R.
\end{align*}
We start with the estimation of the term $T_1$. We get
\begin{align*}
  T_1 \leq \sum_{i , j } |M_{ij}| \ |\cov_{\bar \mu(dx| \omega_t)}  \left( f,  x_i \Delta \omega_j \right)|.
\end{align*}
Because the measure $\mu(dx dz| \omega_t)$ satisfies a uniform LSI by Theorem~\ref{p_mr_OR}, it also satisfies a uniform PI (see Definition~\ref{d_SG}). Then also the marginal $\bar \mu(dz| \omega_t)$ satisfies a uniform PI. Hence we can continue the estimation of $T_1$ according to
\begin{align*}
  T_1 & \leq \sum_{i , j } |M_{ij}| \ |\tilde \omega_j - \omega_j| \ \left( \var_{\bar \mu(dx| \omega_t)}  ( f ) \right)^{\frac{1}{2}} \  \left( \var_{\bar \mu(dx| \omega_t)}  (  x_i  ) \right)^{\frac{1}{2}} \\
& \overset{\PI}{\leq} C(f) \sum_{i , j } |M_{ij}| \ | \tilde \omega_j - \omega_j|  \\
& \leq C(f) \sum_{i , j } \frac{1}{|i-j|^{d + \alpha} +1} \ | \tilde \omega_j - \omega_j| ,
\end{align*}
where we used the decay of interaction~\eqref{e_cond_inter_alg_decay_OR}. Because $i \in \supp (f)$ and $j \notin B_R $ we get
\begin{align*}
  |i-j| \geq \dist \left( \supp (f) , B_R\right).
\end{align*}
Hence, we continue the estimation of $T_1$ as
\begin{align*}
  T_1 & \leq \frac{C(f)}{|\dist \left( \supp (f) , B_R\right)|^{\frac{\alpha}{2}}} \sum_{i , j } \frac{1}{|i-j|^{d + \frac{\alpha}{2}} +1} \ | \tilde \omega_j - \omega_j| \\
& \leq \frac{C(f)}{|\dist \left( \supp (f) , B_R\right)|^{\frac{\alpha}{2}}} \ |\supp(f)| \ \sum_{ j} \frac{1}{|j|^{d + \frac{\alpha}{2}} +1} \ | \tilde \omega_j - \omega_j| .
\end{align*}
Using the fact that 
\begin{align*}
  \sum_{ j} \frac{1}{|j|^{d + \frac{\alpha}{2}} +1} \leq C < \infty
\end{align*}
we get the final form of the estimation of the term $T_1$, namely
\begin{align*}
  T_1 & \leq \frac{C(f)}{|\dist \left( \supp (f) , B_R\right)|^{\frac{\alpha}{2}}} \ \left( \sum_{ j} \frac{1}{|j|^{d + \frac{\alpha}{2}} +1} \ | \tilde \omega_j - \omega_j|^2 \right)^{\frac{1}{2}} .
\end{align*}

Let us turn to the estimation of $T_2$. We have 
\begin{align*}
    T_2 \leq  \left(  \var_{\bar \mu(dx|  \omega_t)} (f) \right)^{\frac{1}{2}} \ \left(  \var_{\bar \mu(dx|  \omega_t)} \left( \int \sum_{l , j } M_{ij} z_l \Delta \omega_j  \mu (dz | x, \omega_t) \right) \right)^{\frac{1}{2}}  .
\end{align*}
Now, an application of the PI yields
\begin{align*}
    T_2 \leq  C(f)  \left( \int \sum_i \left( \frac{d}{dx_i}\int \sum_{l , j } M_{ij} z_l \Delta \omega_j  \mu (dz | x, \omega_t) \right)^2 \bar \mu(dx|  \omega_t)  \right)^{\frac{1}{2}}.
\end{align*}
Straightforward calculation yields
\begin{align*}
\frac{d}{dx_i} & \int \sum_{l , j } M_{ij} z_l \Delta \omega_j  \mu (dz | x, \omega_t) \\
&  = - \cov_{\mu (dz | x, \omega_t)} \left( \sum_{l , j } M_{ij} z_l \Delta \omega_j , \sum_{i,\tilde l} M_{i \tilde l} z_{\tilde l}  \right) \\
&  = - \sum_{l , j } \sum_{i,\tilde l} M_{ij} M_{i \tilde l} \Delta \omega_j \cov_{\mu (dz | x, \omega_t)} \left( z_l , z_{\tilde l}  \right) .
\end{align*}
Using the decay of interaction~\eqref{e_cond_inter_alg_decay_OR} and the decay of correlations~\eqref{e_decay_corr_Otto} we get the estimate
\begin{align*}
&\left| \frac{d}{dx_i} \int \sum_{l , j } M_{ij} z_l \Delta \omega_j  \mu (dz | x, \omega_t) \right| \\
& \qquad   =  \sum_{l , j } \sum_{\tilde l} |M_{ij}| \  |M_{i \tilde l}| \ |\cov_{\mu (dz | x, \omega_t)} \left( z_l , z_{\tilde l}  \right) | \ |\Delta \omega_j| \\
&  \qquad \leq C  \sum_{l , j,  \tilde l} \frac{1}{|i-j|^{d + \alpha}+1} \ \frac{1}{|i-\tilde l|^{d + \alpha}} \  \frac{1}{|l - \tilde l|^{d + \alpha}+1} \ |\Delta \omega_j|.
\end{align*}
Because $i \in \supp (f)$ and $j \notin B_R$ we have
\begin{align*}
  |i-j| \geq \dist \left( \supp (f) , B_R\right).
\end{align*}
This yields the estimate
\begin{align*}
&\left| \frac{d}{dx_i} \int \sum_{l , j } M_{ij} z_l \Delta \omega_j  \mu (dz | x, \omega_t) \right| \\
&  \qquad \leq \frac{C}{| \dist \left( \supp (f) , B_R\right)|^{\frac{\alpha}{2}}} \\
& \qquad \qquad 
\times  \sum_{l , j,  \tilde l} \frac{1}{|i-j|^{d + \frac{\alpha}{2}}+1} \ \frac{1}{|i-\tilde l|^{d + \alpha}} \  \frac{1}{|l - \tilde l|^{d + \alpha}+1} \ |\Delta \omega_j| \\
&  \qquad \leq \frac{C}{| \dist \left( \supp (f) , B_R\right)|^{\frac{\alpha}{2}}}  \  \sum_{  j} \frac{1}{|i-j|^{d + \frac{\alpha}{2}}+1} \ |\Delta \omega_j|.
\end{align*}
Plugging this to the estimation of $T_2$ we get
\begin{align*}
    T_2 &\leq  C(f) \ \frac{C}{| \dist \left( \supp (f) , B_R\right)|^{\frac{\alpha}{2}}} \ \\
  & \qquad \times  \left( \int \sum_i \left(  \sum_{  j} \frac{1}{|i-j|^{d + \frac{\alpha}{2}}+1} \ |\Delta \omega_j| \right)^2 \bar \mu(dx|  \omega_t)  \right)^{\frac{1}{2}} \\
 &\leq  C(f) \ \frac{|\supp(f)|}{| \dist \left( \supp (f) , B_R\right)|^{\frac{\alpha}{2}}} \ \\
  & \qquad \times  \left( \int \left(  \sum_{  j} \frac{1}{|j|^{d + \frac{\alpha}{2}}+1} \ |\Delta \omega_j| \right)^2 \bar \mu(dx|  \omega_t)  \right)^{\frac{1}{2}} .
\end{align*}
We use the fact  
\begin{align*}
  \sum_{  j} \frac{1}{|j|^{d + \frac{\alpha}{2}}+1} \leq C < \infty
\end{align*}
and get
\begin{align*}
      T_2  &\leq C \ C(f) \ \frac{|\supp(f)|}{| \dist \left( \supp (f) , B_R\right)|^{\frac{\alpha}{2}}} \ \\
  & \qquad \times  \left( \int   \sum_{  j} \frac{1}{|j|^{d + \frac{\alpha}{2}}+1} \ |\Delta \omega_j|^2  \bar \mu(dx|  \omega_t)  \right)^{\frac{1}{2}} \\
& =  \frac{C(f)}{| \dist \left( \supp (f) , B_R\right)|^{\frac{\alpha}{2}}}  \left(\sum_{  j} \frac{1}{|j|^{d + \frac{\alpha}{2}}+1} \ |\Delta \omega_j|^2   \right)^{\frac{1}{2}}  .  
\end{align*}
Overall, using the estimation of $T_1$ and $T_2$ this yields the estimate
\begin{align*}
& \left|   \int f \mu (dx | \omega) -   \int f \mu (dx | \tilde \omega) \right| \\ 
& \quad \leq \frac{C(f)}{| \dist \left( \supp (f) , B_R\right)|^{\frac{\alpha}{2}}}   \left(   \sum_{  j} \frac{1}{|j|^{d + \frac{\alpha}{2}}+1} \ |\tilde \omega_j - \omega_j|^2    \right)^{\frac{1}{2}} .
\end{align*}
Now, we get
\begin{align*}
 \int \int &  \left|   \int f \mu (dx | \omega) -   \int f \mu (dx | \tilde \omega) \right|  \bar \mu (d \omega) \bar{\tilde\mu} (d\tilde \omega)  \leq \frac{C(f)}{| \dist \left( \supp (f) , B_R\right)|^{\frac{\alpha}{2}}} \\
& \qquad \times \int \int   \left(   \sum_{  j} \frac{1}{|j|^{d + \frac{\alpha}{2}}+1} \ |\tilde \omega_j - \omega_j|^2    \right)^{\frac{1}{2}}  \bar \mu (d \omega) \bar{\tilde\mu} (d\tilde \omega) \\
& \leq \frac{C(f)}{| \dist \left( \supp (f) , B_R\right)|^{\frac{\alpha}{2}}} \\
& \qquad \times \left(  \sum_{  j} \frac{1}{|j|^{d + \frac{\alpha}{2}}+1}  \left( \int \omega_j^2 \bar \mu (d \omega) + \int \tilde \omega_j^2 \bar{\tilde{ \mu }}(d \tilde \omega) \right) \right)^{\frac{1}{2}} \\
& \leq \frac{C(f)}{| \dist \left( \supp (f) , B_R\right)|^{\frac{\alpha}{2}}},
\end{align*}
where we used in the last step that
\begin{align*}
   \int \omega_j^2 \bar \mu (d \omega) + \int \tilde \omega_j^2 \bar{\tilde{ \mu }} (d \tilde \omega) =  \int \omega_j^2 \mu (d \omega) + \int \tilde \omega_j^2 \tilde{ \mu }(d \tilde \omega) \overset{\eqref{e_sup_moment}}{\leq} C.
\end{align*}

So, we have deduced the desired estimate~\eqref{e_decay_dist_supp}, which closes the argument.
\end{proof}

\section{The criterion of Bakry-\'Emery and the Holley-Stroock perturbation principle}\label{s_BE_HS}

There are a lot of standard criteria to deduce the LSI. The most important criterion is thee \emph{Bakry-\'Emery criterion} connects convexity of the Hamiltonian to the validity of the~PI and the~LSI.
\begin{theorem}[Bakry-\'Emery criterion {\cite[Proposition 3, Corollaire 2]{BE85}}]\label{local:thm:BakryEmery}
 Let $H: D \to \mathbb{R}$ be a Hamiltonian with Gibbs measure $$\mu(dx)=Z_\mu^{-1} \exp\left( -\varepsilon^{-1} H(x) \right) \ dx$$ on a convex domain $D$ and assume that $\nabla^2 H(x) \geq \lambda >0$ for all $x \in \mathbb{R}^n$. Then $\mu$ satisfies LSI with constant $\varrho$ satisfying
 \begin{equation}
  \varrho \geq \frac{\lambda}{\eps} . 
 \end{equation}
\end{theorem}
In non-convex cases the standard tool to deduce the LSI is the \emph{Holley-Stroock perturbation principle}.
\begin{theorem}[Holley-Stroock perturbation principle {\cite[p. 1184]{HS}}]\label{local:thm:HolleyStroock}
  Let $H$ be a Hamiltonian with Gibbs measure $$\mu(dx)=Z_\mu^{-1} \exp\left( -\eps^{-1} H(x) \right) \ dx.$$ Further, let $\tilde H$ denote a bounded perturbation of $H$ and let $\tilde \mu_\varepsilon$ denote the Gibbs measure associated to the Hamiltonian $\tilde H$.
 If $\mu$ satisfies LSI with constant $\varrho$ then also $\tilde\mu$ satisfies the LSI with constant$\tilde \varrho$, where the constants satisfies the bound
 \begin{equation}\label{local:e:HolleyStroockPI-LSI}
\tilde \varrho \geq \exp \left( -\varepsilon^{-1} \osc (H- \tilde H)\right) \varrho,
 \end{equation}
 where $\osc (H - \tilde H) := \sup (H - \tilde H) - \inf (H - \tilde H)$.
\end{theorem}
The perturbation principle of Holley-Stroock~\cite{HS} allows to deduce the LSI constants of non-convex Hamiltonians from the LSI of an appropriately convexified Hamiltonian. However due to its perturbative nature, the dependence of the LSI constant~$\tilde \varrho$ usually is bad in physical parameters like system size or temperature.

\begin{acknowledgement}
The author wants to thank Robin Nittka, Maria Westdickenberg (ne\'e Reznikoff), Felix Otto, Nobuo Yoshida and Chris Henderson for the fruitful and inspiring discussions on this topic. Additionally, the author wants to thank the \emph{Max-Planck Institute for Mathematics in the Sciences} in Leipzig for financial support during the years 2010 to 2012, where most of the content of this article originated.   
\end{acknowledgement}

\bibliographystyle{amsalpha}
\bibliography{or_revisited}

\providecommand{\bysame}{\leavevmode\hbox to3em{\hrulefill}\thinspace}
\providecommand{\MR}{\relax\ifhmode\unskip\space\fi MR }
\providecommand{\MRhref}[2]{%
  \href{http://www.ams.org/mathscinet-getitem?mr=#1}{#2}
}
\providecommand{\href}[2]{#2}
\begin{thebibliography}{GOVW09}

\bibitem[B{\'E}85]{BE85}
D.~Bakry and M.~{\'E}mery, \emph{Diffusions hypercontractives}, S\'eminaire de
  probabilit\'es, {XIX}, 1983/84, Lecture Notes in Math., vol. 1123, Springer,
  Berlin, 1985, pp.~177--206.

\bibitem[BH99]{B-H1}
T.~Bodineau and B.~Helffer, \emph{The log-{S}obolev inequality for unbounded
  spin systems}, J. Funct. Anal. \textbf{166} (1999), no.~1, 168--178.
  \MR{MR1704666}

\bibitem[BHK82]{BHK82}
J.~B{\'e}llissard and R.~H{\o}egh-Krohn, \emph{Compactness and the maximal
  {G}ibbs state for random {G}ibbs fields on a lattice}, Comm. Math. Phys.
  \textbf{84} (1982), no.~3, 297--327.

\bibitem[Diz07]{Dizdar}
D.~Dizdar, \emph{Schritte zu einer optimalen konvergenzrate im hydrodynamischen
  limes der kawasaki dynamik (towards an optimal rate of convergence in the
  hydrodynamic limit for kawasaki dynamics)}, Ph.D. thesis, Diploma thesis,
  Rheinische Friedrich-Wilhelms-Universit\"{a}t Bonn, 2007.

\bibitem[GOVW09]{GORV}
N.~Grunewald, F.~Otto, C.~Villani, and M.~Westdickenberg, \emph{A two-scale
  approach to logarithmic {S}obolev inequalities and the hydrodynamic limit},
  Ann. Inst. H. Poincar\'e Probab. Statist. \textbf{45} (2009), no.~2,
  302--351. \MR{MR2521405}

\bibitem[Gro75]{Gross}
L.~Gross, \emph{Logarithmic {S}obolev inequalities}, Amer. J. Math. \textbf{97}
  (1975), 1061--1083. \MR{MR0420249}

\bibitem[HM79]{HorMor}
T.~Horiguchi and T.~Morita, \emph{Upper and lower bounds to a correlation
  function for an {I}sing model with random interactions}, Phys. Lett. A
  \textbf{74} (1979), no.~5, 340--342. \MR{591328}

\bibitem[HS87]{HS}
R.~Holley and D.~Stroock, \emph{Logarithmic {S}obolev inequalities and
  stochastic {I}sing models}, J. Statist. Phys. \textbf{46} (1987), no.~5-6,
  1159--1194. \MR{893137}

\bibitem[Led01]{L}
M.~Ledoux, \emph{Logarithmic {S}obolev inequalities for unbounded spin systems
  revisted}, Sem. Probab. XXXV, Lecture Notes in Math., Springer \textbf{1755}
  (2001), 167--194.

\bibitem[Men14]{Cov_est}
G.~Menz, \emph{A {B}rascamp-{L}ieb type covariance estimate}, arXiv:1402.5160.

\bibitem[MN13]{MN}
G.~Menz and R.~Nittka, \emph{Decay of correlations in 1d lattice systems of
  continuous spins and long-range interaction}, arXiv:1309.0857.

\bibitem[MO94]{MaOl}
F.~Martinelli and E.~Olivieri, \emph{Approach to equilibrium of glauber
  dynamics in the one phase region - i. the attractive case}, Communications in
  Mathematical Physics \textbf{161} (1994), no.~3, 447--486, cited By (since
  1996)89.

\bibitem[MO13]{MO}
G.~Menz and F.~Otto, \emph{Uniform logarithmic {S}obolev inequalities for
  conservative spin systems with super-quadratic single-site potential}, Annals
  of Probability \textbf{41} (2013), no.~3B, 21820--2224.

\bibitem[OR07]{OR07}
F.~Otto and M.~G. Reznikoff, \emph{A new criterion for the logarithmic
  {S}obolev inequality and two applications}, J. Funct. Anal. \textbf{243}
  (2007), no.~1, 121--157.

\bibitem[PS01]{ProSco}
A.~Procacci and B.~Scoppola, \emph{On decay of correlations for unbounded spin
  systems with arbitrary boundary conditions}, J. Statist. Phys. \textbf{105}
  (2001), no.~3-4, 453--482. \MR{1871653}

\bibitem[Roy07]{Roy07}
G.~Royer, \emph{An initiation to logarithmic {S}obolev inequalities}, SMF/AMS
  Texts and Monographs, vol.~14, American Mathematical Society, Providence, RI,
  2007, Translated from the 1999 French original by Donald Babbitt.

\bibitem[Syl76]{Sylvester}
G.~S. Sylvester, \emph{Inequalities for continuous-spin {I}sing ferromagnets},
  J. Statist. Phys. \textbf{15} (1976), no.~4, 327--341. \MR{0436856}

\bibitem[SZ92a]{StrZeg}
D.~W. Stroock and B.~Zegarlinski, \emph{The logarithmic sobolev inequality for
  continuous spin systems on a lattice}, Journal of Functional Analysis
  \textbf{104} (1992), no.~2, 299--326, cited By (since 1996)29.

\bibitem[SZ92b]{StrZeg2}
\bysame, \emph{The logarithmic sobolev inequality for discrete spin systems on
  a lattice}, Communications in Mathematical Physics \textbf{149} (1992),
  no.~1, 175--193, cited By (since 1996)65.

\bibitem[Yos99]{Yos99}
N.~Yoshida, \emph{The log-{S}obolev inequality for weakly coupled lattice
  fields}, Probab. Theory Related Fields \textbf{115} (1999), no.~1, 1--40.

\bibitem[Yos01]{Yos01}
\bysame, \emph{The equivalence of the log-{S}obolev inequality and a mixing
  condition for unbounded spin systems on the lattice}, Ann. Inst. H.
  Poincar\'e Probab. Statist. \textbf{37} (2001), no.~2, 223--243.

\bibitem[Yos03]{Yos_2}
\bysame, \emph{Phase transition from the viewpoint of relaxation phenomena},
  Rev. Math. Phys. \textbf{15} (2003), no.~7, 765--788. \MR{2018287}

\bibitem[Zeg96]{Zeg96}
B.~Zegarlinski, \emph{The strong decay to equilibrium for the stochastic
  dynamics of unbounded spin systems on a lattice}, Comm. Math. Phys.
  \textbf{175} (1996), no.~2, 401--432.

\bibitem[Zit08]{Zitt}
P.-A. Zitt, \emph{Functional inequalities and uniqueness of the {G}ibbs
  measure---from log-{S}obolev to {P}oincar\'e}, ESAIM Probab. Stat.
  \textbf{12} (2008), 258--272. \MR{2374641}

\end{thebibliography}

\end{document}